\documentclass[12pt]{amsart}
\usepackage[margin=.75in]{geometry}
\usepackage{amsmath}
\usepackage{amsxtra}
\usepackage{amscd}
\usepackage{amsthm}
\usepackage{amssymb}
\usepackage[normalem]{ulem}
\usepackage{comment}
\usepackage{color}
\usepackage{tikz}
\usetikzlibrary{matrix,arrows,decorations.pathmorphing,automata, matrix, positioning, calc, shapes.multipart}
\usepackage{tikz-cd}
\usepackage{graphicx}
\usepackage{hyperref}
\usepackage[shortlabels]{enumitem}
\usepackage{float}
\usepackage{epigraph}
\usepackage{mathrsfs}
\usepackage{cleveref}

\newcommand{\fS}{\mathfrak{S}}

\newcommand{\fc}{\mathfrak{c}}

\newcommand{\fm}{\mathfrak{m}}

\newcommand{\cA}{\mathcal{A}}

\newcommand{\cC}{\mathcal{C}}
\newcommand{\cD}{\mathcal{D}}
\newcommand{\cE}{\mathcal{E}}
\newcommand{\cF}{\mathcal{F}}

\newcommand{\cO}{\mathcal{O}}
\newcommand{\cP}{\mathcal{P}}

\newcommand{\cS}{\mathcal{S}}

\newcommand{\Z}{\mathbb{Z}}

\newcommand{\Q}{\mathbb{Q}}
\newcommand{\C}{\mathbb{C}}

\newcommand{\br}{\mathbf{r}}
\newcommand{\bs}{\mathbf{s}}
\newcommand{\bt}{\mathbf{t}}

\newcommand{\bv}{\mathbf{v}}

\newcommand{\be}{\beta}
\newcommand{\si}{\sigma}
\newcommand{\la}{\lambda}

\newcommand{\ka}{\kappa}

\newcommand{\de}{\delta}

\newcommand{\bla}{\boldsymbol{\la}}
\newcommand{\bnu}{\boldsymbol{\nu}}
\newcommand{\bmu}{\boldsymbol{\mu}}
\newcommand{\bpi}{\boldsymbol{\pi}}

\newcommand{\ta}{{\tilde{a}}}

\newcommand{\tf}{{\tilde{f}}}

\newcommand{\dbs}{\dot{\bs}}
\newcommand{\dbr}{\dot{\br}}

\newcommand{\tdf}{\tilde{\dot{f}}}

\newcommand{\lra}{\longrightarrow}
\newcommand{\Ra}{\Rightarrow}

\newcommand{\Ue}{\mathcal{U}_t (\widehat{\mathfrak{sl}_e})}
\newcommand{\Ul}{\mathcal{U}_{-1/t} (\widehat{\mathfrak{sl}_\ell})}

\newcommand{\sle}{\widehat{\mathfrak{sl}_e}}
\newcommand{\sll}{\widehat{\mathfrak{sl}_\ell}}
\newcommand{\bemptyset}{\boldsymbol{\emptyset}}

\newcommand{\ds}{\displaystyle}

\newcommand{\Irr}{\text{\rm Irr}}

\newcommand{\id}{\text{\rm Id}}

\newcommand{\ug}{\operatorname{Ug}}
\newcommand{\kl}{{\operatorname{Kl}}}

\numberwithin{equation}{section}

\theoremstyle{plain}
\newtheorem{theorem}[equation]{Theorem}
\newtheorem{lemma}[equation]{Lemma}
\newtheorem{proposition}[equation]{Proposition}
\newtheorem{corollary}[equation]{Corollary}

\newtheorem{definition}[equation]{Definition}

\theoremstyle{remark}
\newtheorem{example}[equation]{Example}
\newtheorem{remark}[equation]{Remark}

\newcommand{\emp}{\emptyset}
\newcommand{\bemp}{\bemptyset}

\newcommand{\slell}{\widehat{\mathfrak{sl}_\ell}}
\newcommand{\slinf}{\mathfrak{sl}_\infty}

\newcommand{\rev}{\mathrm{rev}}

\newcommand{\WC}{\mathsf{WC}_{-\leftarrow +}}
\newcommand{\wc}{\mathsf{wc}_{-\leftarrow +}}

\usepackage[foot]{amsaddr}

\title{Generalized Mullineux involution and perverse equivalences}
\author{Thomas Gerber}
\address[T.G.]{\'Ecole Polytechnique F\'ed\'erale de Lausanne, 1015 Lausanne, Switzerland.}
\email{thomas.gerber@epfl.ch}
\author{Nicolas Jacon}
\address[N.J.]{Universit\'e de Reims Champagne Ardennes, Laboratoire de Math\'ematiques CNRS UMR 9008,
Moulin de la Housse BP 1039, 51100 Reims, France}
\email{nicolas.jacon@univ-reims.fr}
\author{Emily Norton}
\address[E.N.]{University of Bonn, Endenicher Allee 60, 53115 Bonn, Germany}
\email{enorton@mpim-bonn.mpg.de}

\begin{document}

\begin{abstract}
We define a generalization of the Mullineux involution on multipartitions
using the theory of crystals for higher level Fock spaces.
Our generalized Mullineux involution turns up in representation theory via two important derived functors on cyclotomic Cherednik category $\cO$:
Losev's ``$\ka=0$'' wall-crossing, and Ringel duality.
\end{abstract}

\maketitle

\sloppy

\section*{Introduction}

It has been known since the foundational work of Frobenius that partitions of $n$ naturally label the complex irreducible representations of the symmetric group $\mathfrak{S}_n$. If we take an irreducible representation labeled by a partition $\lambda$ and tensor it with the sign representation, we obtain an irreducible representation labeled by the transpose of $\lambda$. 
The story in positive characteristic is more subtle: the irreducible representations of $\mathfrak{S}_n$ over a field of characteristic $p>0$ are labeled by the $p$-regular partitions (partitions in which each non-zero part occurs at most $p-1$ times).
Tensoring such a representation with the sign representation still yields an irreducible representation, but the resulting involution on $p$-regular partitions lacks such a simple description as taking the transpose. In 1979, Mullineux defined a combinatorial algorithm producing an involution on $p$-regular partitions (now called the Mullineux involution), and he conjectured that this involution describes the result of tensoring an irreducible representation with the sign representation in characteristic $p$ \cite{Mullineux1979}.
 
In 1995, Kleshchev came up with a surprising algorithm to compute the Mullineux involution \cite{Kleshchev1995III}. 
In fact, whereas Mullineux's algorithm involved repeated operations with strips of boxes in the rim of the Young diagram, 
it has been understood later that Kleshchev's algorithm
can be interpreted in terms of the Kashiwara crystal of an irreducible highest weight module of level $1$ for the quantum group of affine type $A_{p-1}$ \cite{LLT1996}:
the Mullineux involution is the automorphism of oriented $\Z/p\Z$-colored graphs which switches the sign of each arrow.
This algorithm
led to Ford and Kleshchev's proof of the Mullineux Conjecture \cite{FK}; a different proof was given later by Bessenrodt and Olsson \cite{BessenrodtOlsson1998}.

The Mullineux involution can be generalized to various extents. 
First, one can look at the Hecke algebra of $\mathfrak{S}_n$ (which can be seen as a deformation of the group algebra) with parameter specialized to a primitive $e$-th root of $1$, $e\in\Z_{\geq 2}$. 
An involution on the set of $e$-regular partitions  (which parametrize the associated irreducible representations) can then be defined using crystals as above, see \cite[Section 7]{LLT1996}. 
Next, Fayers defined a Mullineux involution for the Hecke algebra of the complex reflection group $G(\ell,1,n)$ (the Ariki-Koike algebra) \cite{fayers}. 
Fayers' involution can also be computed using crystal graphs (now for irreducible highest weight modules of level $\ell$) or via a combinatorial algorithm generalizing Mullineux's original procedure  \cite{JaconLecouvey2008}. The Ariki-Koike algebra has cell modules labeled by all $\ell$-partitions, but simples labeled only by Uglov $\ell$-partitions
(which coincide with $e$-regular partitions for $\ell=1$).
However, its module category is a quotient of a highest weight category $\cO_{\kappa,\bs}$ where every $\ell$-partition labels a simple module, raising the question whether the Mullineux involution admits a further meaningful extension to that bigger category.

Namely, consider the category $\cO_{\kappa,\bs}(n)$ of the Cherednik algebra of $G(\ell,1,n)$. This category depends on parameters 
$\kappa\in\Q^\times$ and $\bs\in\Q^\ell$ \cite{GGOR2003}, \cite{Rouquier2008}, \cite{Losev2015a} and its Grothendieck group has a basis consisting of $\ell$-partitions of $n$.
In order to relate categories depending on different parameters, Losev introduced derived equivalences called wall-crossing functors \cite{Losev2015}. Each wall-crossing can be thought of as a partial version of a duality functor called Ringel duality. 
The wall-crossing functors and  Ringel duality are examples of a special kind of derived equivalence called a perverse equivalence \cite{ChuangRouquier2008},\cite{Losev2017}, and consequently they effect a permutation of the set of simple objects, that is, a permutation of $\ell$-partitions.
It is natural to ask for an explicit formula for these combinatorial maps.

We now summarize the main results of this paper. In Theorem \ref{genmul} we define a generalization of the Mullineux involution on all multipartitions. The proof uses the result of \cite{Gerber2016} that the $\sle$-, $\slinf$-, and $\slell$-crystals on the level $\ell$ Fock space all commute.
Our involution $\Phi$ is compatible 
with both Fayers' and Losev's involutions, recovering Fayers' in the case of Uglov multipartitions. The next question is the representation-theoretic meaning of $\Phi$. In Section \ref{chered} we study the combinatorics of perverse equivalences on module categories of Cherednik algebras. Theorems \ref{wcmul1} and \ref{wcmul2} give some formulas 
for the $\kappa=0$ wall-crossing in terms of $\ell$ copies of the level $1$ Mullineux involution; we recover
\cite[Corollary 5.7]{Losev2015} when $\ell=1$. 
Next, we look for a duality functor which produces the involution $\Phi$, and we find in Theorem \ref{ringelmul} that $\Phi$ arises from  Ringel duality. Here the perspective of diagrammatic Cherednik algebras \cite{Webster2017} is crucial, especially \cite[Corollary 5.11]{Webster2017}.
In Section \ref{perspectives}, we define a refinement of $\Phi$ with a speculative eye towards the Alvis-Curtis duality, a perverse equivalence for finite groups of Lie type which still lacks a combinatorial description outside type $A$. This generalizes Dudas and the second author's definition of a generalized Mullineux involution in the case $\ell=1$ \cite{DudasJacon2018} by refining the $\slinf$-crystal with respect to an integer parameter $d$.

\medskip

{\bf Acknowledgements.} The authors thank Olivier Dudas, Ivan Losev, Tomasz Przezdziecki, Catharina Stroppel, and  Ben Webster for useful discussions, and the anonymous referee for helpful suggestions to improve the readability of the paper.
The first author is supported by the Ambizione project of the Swiss National Science Foundation.
The  second author is supported by Agence Nationale de la Recherche  GeRepMod ANR-16-CE40-0010-01.

\section{The Mullineux involution for cyclotomic Hecke algebras}\label{mulAK}

We here give a quick review of the definition of the Mullineux involution
for cyclotomic Hecke algebras and its crystal interpretation \cite{fayers}, \cite{JaconLecouvey2008}. 
This generalizes the usual notion of Mullineux involution.

\subsection{Definition}

Let $\ell\in\Z_{\geq1}$ and $n\in\Z_{\geq1}$.
Denote $W_{\ell,n}$  the complex reflection group $G(\ell,1,n)=\fS_n\ltimes(\Z/\ell\Z)^n$.
Let $R$ be a field of arbitrary characteristic and let   $v\in R^{\times}$ 
and let $(s_{1},s_{2},\ldots,s_{\ell})$ be an $\ell$-tuple of integers. 

The cyclotomic Hecke algebra (also called Ariki-Koike algebra) $\mathcal{H}_{R,n}^\bs=\mathcal{H}(v;s_1,\ldots,s_\ell)$ over $R$ 
is the unital associative $R$-algebra with a presentation by

\begin{itemize}
\item  generators: $T_0$, $T_1$,\ldots, $T_{n-1}$,

\item  relations: 
\begin{align*}
& T_0 T_1 T_0 T_1=T_1 T_0 T_1 T_0, \\
& T_iT_{i+1}T_i=T_{i+1}T_i T_{i+1}\ (i=1,\ldots,n-2), \\
& T_i T_j =T_j T_i\ (|j-i|>1), \\
&(T_0-v^{s_1})(T_0-v^{s_2})\ldots(T_0- v^{s_\ell}) = 0, \\
&(T_i-v)(T_i+1) = 0\ (i=1,\ldots,n-1).
\end{align*}
\end{itemize}

It can be seen as a deformation of the group algebra of $W_{\ell,n}$.
In particular, if  $\ell=1$, it is the usual  Hecke algebra of type $A$ and if moreover $v=1$, we obtain the group algebra $R\mathfrak{S}_n$ of the symmetric group. We denote by:
\begin{itemize}
\item  $\Pi^\ell$ the set of all $\ell$-partitions, that is, the set of all $\ell$-tuples $(\lambda^{1},\ldots,\lambda^\ell)$ of partitions.
\item $\Pi=\Pi^1$ the set of all partitions.
\end{itemize}
The unique $\ell$-partition of size $0$ is denoted by $\bemptyset$. 
For any subset $\cE$ of $\Pi^\ell$ and any $n\in\Z_{\geq0}$,
we denote  by $\cE(n)$ the set of $\ell$-partitions in $\cE$ of total size $|\lambda^1|+\ldots+|\lambda^\ell|=n$. 

Let $e$ be the multiplicative order of  $v$ in $R$.  We assume that  $v\neq 1$ so that we have $e\in \{2,3,\ldots \}\sqcup \{ \infty\}$.  
We now recall several facts about the representation theory of cyclotomic Hecke algebras. We refer to \cite[Chapter 5]{GeckJacon2011} for details. For each 
$\bla\in \Pi^{\ell} (n)$, there is  an $\mathcal{H}_{R,n}^\bs$-module $S^{\bla}$ which is the Specht module associated to $\bla$. 
There exists a natural bilinear form, $\mathcal{H}_{R,n}^\bs$-invariant, on each of these modules and an associated
radical such that the quotients $D^{\bla}:=S^{\bla}/\text{rad}(S^{\bla})$ are either $0$ or
irreducible. The non-zero $D^{\bla}$ then give a 
complete set of non-isomorphic simple $\mathcal{H}_{R,n}^\bs$-modules.

The set $\left\{\bla\in\Pi^\ell(n) \mid D^{{\boldsymbol{\lambda}}}\neq 0\right\}$ depends only on $e$ and $\bs$
and is known as the set of Kleshchev $\ell$-partitions, denoted $\kl_{e,\bs} (n)$.
It was originally defined using the notion of crystal (see \cite[Section 6.2.10]{GeckJacon2011}), but there is another independent description, see \cite{Jkle}.  

\begin{remark} Let $e\geq 2, \bs=(s_1,\ldots,s_\ell)$ and $\bt=(t_1,\ldots,t_\ell)$ such that $t_i=s_i\mod e$ for all $i=1,\ldots, \ell$. 
Observe that for any $n\in\Z_{\geq0}, \mathcal{H}_{R,n}^\bs = \mathcal{H}_{R,n}^{\bf t}$ and the definition
 of Kleshchev $\ell$-partitions gives that   $\kl_{e,\bs}(n)=\kl_{e,\bt}(n)$. Now if there exists $\sigma\in \mathfrak{S}_\ell$ 
  such that  $t_i=s_{\sigma(i)}\mod e$ for all $i=1,\ldots, \ell$ then we still have that for any $n\in\Z_{\geq0}, \mathcal{H}_{R,n}^\bs = \mathcal{H}_{R,n}^{\bf t}$
   but  $\kl_{e,\bs}(n)$ is different from $\kl_{e,\bt}(n)$ in general.
\end{remark}

Set $\widetilde{\mathcal{H}}_{R,n}^\bs:=\mathcal{H}_{R,n}(v^{-1};s_{\ell},\ldots,s_{1})$ and denote by   $\widetilde{T
}_{0}$,\ldots,$\widetilde{T}_{\ell-1}$ the associated standard generators.
 For each $\bla \in \Pi^\ell (n)$, denote by   $\widetilde{S}^{{\boldsymbol{\lambda}}}$   the associated Specht module 
 of $\widetilde{\mathcal{H}}_{R,n}^\bs$. By \cite{fayers}, the simple modules of $\widetilde{\mathcal{H}}_{R,n}^\bs$ are labeled by the set  $\kl_{e,-\bs_{\textrm{rev} }} (n)$
  where $-\bs_{\textrm{rev}}=(-s_\ell,\ldots,-s_1)\in (\mathbb{Z}/e\mathbb{Z})^\ell$. Thus, 
  for each 
 $ \bla\in \kl_{e,-\bs_{\textrm{rev} }} (n)$,  we have an associated simple        $\widetilde{\mathcal{H}}_{R,n}^\bs$-module  $\widetilde{D}^{{\boldsymbol{\lambda}}
}$. We have an involutive isomorphism $\theta :
\mathcal{H}_{R,n}^\bs\rightarrow \widetilde{\mathcal{H}}_{R,n}^\bs$ given by 
\begin{equation*}
T_{0}\mapsto \widetilde{T}_{0}\qquad T_{i}\mapsto -v\widetilde{T}_{i}\
(i=1,\ldots,n-1).
\end{equation*}
Then, $\theta $ induces a functor $F$ from the category of $\widetilde{
\mathcal{H}}_{R,n}^\bs$-modules to the category of ${\mathcal{H}}_{R,n}^\bs$-modules. As a consequence, we obtain a bijective map 

\begin{equation*}
\mathfrak{m}_{e,{\bs}}:\kl_{e,\bs} (n) \rightarrow \kl_{e,-{\bs_{\textrm{rev}}}} (n),
\end{equation*}
satisfying 
\begin{equation*}
F(\widetilde{D}^{\mathfrak{m}_{e,{\bs}} ({\boldsymbol{\lambda}})})\simeq {D}^{
\boldsymbol{\lambda}},
\end{equation*}
for all $\lambda \in \kl _{e,\bs}$. 

\begin{remark}\label{rem_mulFayers}
\begin{enumerate}
\item By definition of $\theta$, we have
$\fm_{e,-\bs_\rev}\circ\fm_{e,\bs}=\id_{\kl_{e,\bs}}$
and
$\fm_{e,\bs}\circ\fm_{e,-\bs_\rev}=\id_{\kl_{e,-\bs_\rev}}$.
\item Assume that $\ell=1$. 
Then the map $\mathfrak{m}_{e,\bs}$ is an involution and it does not depend on the choice of $\bs$.
In fact, we have $\fm_{e,\bs}=m_e$, where $m_e$ is the usual Mullineux involution defined in the introduction. 
\end{enumerate}
\end{remark}

\subsection{The quantum algebra $\Ue$}

We denote by $\Lambda_0, \ldots, \Lambda_{e-1}$ (where the subscripts are understood modulo $e$) the fundamental weights attached to the 
Kac-Moody algebra $\sle$. The simple roots are denoted by $\alpha_0,\ldots,\alpha_{e-1}$ 
and $\delta:=\alpha_0+\ldots +\alpha_{e-1}$ is the null root. The fundamental weights and the simple roots are related by the following formula:
$$\alpha_i=2\Lambda_i-\Lambda_{i-1}-\Lambda_{i+1}+\delta_{i,0}\delta \quad \text{ for all } 0\leq i\leq e-1$$
(where $\delta_{ij}$ denotes the Kronecker symbol).    
We denote by $\mathcal{P}=\bigoplus_{0\leq i\leq e-1}  \Z\Lambda_i \oplus \mathbb{Z} \delta$ the weight lattice and by  $\Ue$ the quantum algebra associated to $\sle$, where $t$ is  an indeterminate.  This is an algebra over  $\C (t)$  with  generators $e_i$, $f_i$, $t_i^{\pm 1}$ ($0\leq i\leq e-1$) and $\partial^{\pm1}$ subject to standard relations which we do not recall.  
We refer to \cite[Chapter 6]{GeckJacon2011} for details on this algebra and its the representation theory.

\subsection{The level $\ell$ Fock space} \label{fock}

Let us fix some notation.  Fix $e,\ell\geq 2$ and $s\in\Z$. For $\mathbb{K}=\Z$ or $\Q$, we 
denote 
$$\mathbb{K}^\ell(s)=\left\{ (s_1,\ldots,s_\ell)\in\mathbb{K}^\ell \, \left| \,  \sum_{i=1}^\ell s_i=s\right. \right\}.$$
For $\bs\in\Z^\ell(s)$, we denote by $\Pi^\ell_\bs$ 
the set of all symbols of the form $|\bla,\bs\rangle$ with $\bla \in \Pi^\ell$. 
Further, denote by $\Pi^\ell_s$ the sets of all elements in $\Pi^\ell_\bs$ where $\bs\in \Z^\ell(s)$.
Let 
$\cF_{e,\bs}$ be the 
$\C (t)$-vector space
with standard basis $\Pi^\ell_\bs$,
i.e. $\cF_{e,\bs}=\bigoplus_{\bla\in\Pi^\ell}\C (t) |\bla,\bs\rangle$, called the Fock space of level $\ell$ and rank $e$
(associated to the charge $\bs$). This space can be endowed with a structure of  an integrable $\Ue$-module, see  \cite[Section 6.2]{GeckJacon2011}.

One can decompose this module as a direct sum of remarkable vector spaces. Indeed, if $w:=\sum_{0\leq i\leq e-1} a_i \Lambda_i+d\delta \in \mathcal{P}$,  
define:
$$\cF_{e,\bs}[w]:= \{m\in \cF_{e,\bs}\ |\ \partial m=t^dm,\ t_i m=t^{a_i} m\ \forall i\in [0,e-1]\}.$$
If this space is non zero, we say that $w$ is a weight for $\cF_{e,\bs}$ and $\cF_{e,\bs}[w]$ is called the $w$-weight space.
The elements of $\cF_{e,\bs}[w]$ are called weight vectors.  
Importantly, each 
element of the standard basis  $|\bla,{\bf s} \rangle$ is a weight vector and the associated weight may be easily computed (see for example \cite[Corollary 2.5]{Yvonne2007}).
In particular, one can always write it as follows:
$$d \delta +\sum_{0\leq i\leq e-1} \Lambda_{s_i} -\sum_{0\leq i\leq e-1} m_i \alpha_i$$
and the number $\sum_{0\leq i\leq e-1} m_i$ corresponds to the size of $\bla$. 
In particular, the weight of $|\bemp,\bs\rangle$ is $\sum_{0\leq i\leq e-1} \Lambda_{s_i}$.
Thus  $\cF_{e,\bs}$  is  the direct sum of its weight spaces.

\subsection{The $\sle$-crystal of the Fock space} \label{fock_crys}

As mentioned in the introduction, an important part of the representation theory of cyclotomic Hecke algebras 
is controlled by the theory of crystals for Fock spaces. 
The $\sle$-crystal  of the Fock space $\cF_{e,\bs}$ is a combinatorial construction
arising from the action of  $\Ue$ on the   Fock space (see the general definition in \cite{Kashiwara1991}, \cite{HongKang2002}). 
Concretely, the $\sle$-crystal is a graph with 
\begin{itemize}
\item {vertices}: the elements of $\Pi^\ell_\bs$,
\item {arrows}:  $|\bla,\bs\rangle \overset{i}{\rightarrow }  |\bmu,\bs\rangle$ for $\bla,\bmu\in\Pi^\ell$, $i\in\{0,\ldots,e-1\}$
if and only if $|\bmu,\bs\rangle = \widetilde{f}_i  |\bla,\bs\rangle$, where $\widetilde{f}_i$ is the $i$-th lowering Kashiwara operator of $\Ue$.
\end{itemize}
An explicit recursive formula for computing the $\sle$-crystal is given in \cite{JMMO1991} in terms of adding good boxes, see also \cite{FLOTW1999}.
It has 
infinitely many connected components, each of which is parametrized by its unique source vertex,
called a highest weight vertex.
We denote by $\ug_{e,\bs}$ the $\ell$-partitions appearing in the connected component parametrized by the highest weight vertex $\bemp=(\emp,\ldots,\emp)$,
and call them the Uglov $\ell$-partitions.
When $\ell=1$ this set is nothing but the set of $e$-regular partitions. 
The following is an easy consequence of the definition of Kleshchev and Uglov $\ell$-partitions (see \cite[Ex. 6.2.16]{GeckJacon2011}).

\begin{proposition} \label{ariki}
Fix $n\in\Z_{\geq 0}$. Let $\bs=(s_1,\ldots,s_\ell)\in \mathbb{Z}^\ell$ and
${\bf t}=(t_1,\ldots,t_\ell)\in \mathbb{Z}^\ell$ be such that $s_i=t_i\mod e$ for all $i=1,\ldots,\ell$,
and  $t_i-t_{i-1}>n-1$ for all $i=2,\ldots,\ell$.
Then 
$$\ug_{e,\bt}(n)=\kl_{e,\bt}(n)=\kl_{e,\bs}(n).$$
 \end{proposition}
In other words, Kleshchev $\ell$-partitions are a particular case of Uglov $\ell$-partitions,
i.e. we can index irreducible modules of cyclotomic Hecke algebras by certain
vertices of Fock space crystals.
The following result is due to Fayers \cite[Section 2]{fayers} in the case of Kleshchev multipartitions
(that is, under the condition of \Cref{ariki}) and to \cite[Section 4]{JaconLecouvey2008} in general.

\begin{theorem}\label{FJLmul} Let $n\in\Z_{\geq0}, \bs\in\Z^\ell(s)$ and $e\geq 2$.
 There exists a unique bijection
$$\begin{array}{rccc}
{\tt\Phi}_{e,\bs} :& \ug_{e,\bs}(n) &\lra& \ug_{e,-\bs_\rev}(n)\\
& \bla & \longmapsto & {\tt\Phi}_{e,\bs}(\bla)
\end{array}
$$
such that 
\begin{itemize}
\item ${\tt\Phi}_{e,\bs}(\bemp)=\bemp$,
 \item for all $0\leq i\leq e-1$, we have ${{\tt\Phi}_{e,\bs} }\circ \widetilde{f}_i =\widetilde{f}_{-i}  \circ {{\tt\Phi}_{e,\bs} }$.
\end{itemize}
\end{theorem}

This means that for all paths
\begin{equation*}
|\boldsymbol{\emptyset},\bs\rangle\overset{i_{1}}{\rightarrow }\cdot \overset{i_{2}}{
\rightarrow }\cdot \overset{i_{3}}{\rightarrow }\cdots \overset{i_{n}}{
\rightarrow }|{\boldsymbol{\lambda}},\bs\rangle
\end{equation*}
in the  $\sle$-crystal on the Fock space $\cF_{e,\bs}$, there exists a
corresponding path 
\begin{equation*}
|\boldsymbol{\emptyset},-\bs_\rev\rangle\overset{-i_{1}}{\rightarrow }\cdot \overset{-i_{2}}{
\rightarrow }\cdot \overset{-i_{3}}{\rightarrow }\cdots \overset{-i_{n}}{
\rightarrow }|{\boldsymbol{\mu}},-\bs_\rev\rangle
\end{equation*}
in the $\sle$-crystal on the Fock space $\cF_{e,-\bs_{\textrm{rev}}}$  from the empty 
$\ell$-partition to an $\ell$-partition ${\boldsymbol{\mu}}\in \ug_{e,-\bf s_{\textrm{rev}}}$.
Then ${\tt\Phi}_{e,\bs}(\bla)=\bmu$.
In \cite{JaconLecouvey2008}, it is explained how
the map ${\tt\Phi}_{e,\bs}$ can be explicitly computed without constructing the $\sle$-crystal.

\begin{example}
Take $s=4$, $e=4$, $\ell=3$, $\bs=(5,-1,0)$ 
(so that $-\bs_\rev=(0,1,-5)$) and $\bla=(1,3.2,\emp)$\footnote{In the examples, we use the multiplicative notation for partitions
and we forget the brackets around components of a multipartition.}.
One can write for instance $\bla = \tf_1\tf_1\tf_3\tf_0\tf_2\tf_3 \, \bemp$,
so that $\bla\in\ug_{e,\bs}(6)$.
Therefore, in the crystal of the Fock space $\cF_{e,-\bs_\rev}$, we get
\begin{align*}
{\tt\Phi}_{e,\bs} &= \tf_{-1}\tf_{-1}\tf_{-3}\tf_0\tf_{-2}\tf_{-3} \, \bemp \\ 
& =\tf_3\tf_3\tf_1\tf_0\tf_2\tf_1 \, \bemp  \\
& = (2.1,3,\emp).
\end{align*}
\end{example}

The following result by Fayers \cite{fayers} gives the desired crystal interpretation of the Mullineux involution for cyclotomic Hecke algebras.

\begin{theorem}[Fayers] Fix $n\in\Z_{\geq 0}, \bs\in\Z^\ell$ and $e\geq 2$.
For all $\bla\in\kl_{e,\bs}(n)$, we have
\begin{equation*}
\mathfrak{m}_{e,{\bs}} ({\boldsymbol{\lambda}})={\tt\Phi}_{e,\bs}  (\bla).
\end{equation*}

\end{theorem}

To summarize, starting with the usual Mullineux involution $m_e$ for the symmetric group, we obtain:
 \begin{itemize}
 \item a generalization of $m_e$ : the involution $\mathfrak{m}_{e,{\bs}}$ on the set of Kleshchev $\ell$-partitions which label the irreducible representations of cyclotomic Hecke algebras. 
 If $\ell=1$, we have $\mathfrak{m}_{e,{\bs}}=m_e$.
  \item a generalization of $\mathfrak{m}_{e,{\bs}}$ : the involution ${\tt\Phi}_{e,\bs}$ on the set of Uglov  $\ell$-partitions. 
  If $\bs$ is such that  $s_i-s_{i-1}>n-1$ for all $i=2,\ldots,\ell$, we have ${\tt\Phi}_{e,\bs}=\mathfrak{m}_{e,{\bs}}$.
\end{itemize}

\section{The generalized Mullineux involution}\label{section_genmul}

Remember from Section \ref{fock} that we have fixed $e,\ell\geq 2$ and $s\in\Z$.
Let us denote   $\cF_{e,s}=\bigoplus_{\bs\in\Z^\ell(s)}\cF_{e,\bs}$
and $\cF_e= \bigoplus_{s\in\Z}\cF_{e,s}$.

\subsection{Triple crystal structure}

By Section \ref{fock}, the space $\cF_{e,s}$ has a structure of integrable $\Ue$-module of level $\ell$. 
This space can  also be endowed with a structure of $\Ul$-module of level $e$.
Denote by $\dot{\Lambda}_0$, \ldots, $\dot{\Lambda}_{l-1}$ the fundamental weights attached to the 
Kac-Moody algebra $\sll$. The simple roots are denoted by $\dot{\alpha}_0, \ldots, \dot{\alpha}_{\ell-1}$ 
and $\dot{\delta}$ is the null root.    
We denote by $\dot{\mathcal{P}}=\bigoplus_{0\leq i\leq \ell-1}  \Z\dot{\Lambda}_i \oplus \mathbb{Z} \dot{\delta}$ the corresponding weight lattice.

Following \cite{Gerber2016}, there is a \textit{level-rank duality} between
$\ell$-partitions and $e$-partitions.
This is a map
$$
\begin{array}{rrcl}
\mathsf{k}_{s}^{\ell,e}: & \Pi^\ell_s & \lra & \Pi^e_{-s} \\
\end{array}
$$
inducing a linear map between the Fock spaces $\cF_{e,s}\lra\cF_{\ell,-s}$. To avoid cumbersome notations, write $\mathsf{k}$ for $\mathsf{k}_{s}^{\ell,e}$ and $\dot{\mathsf{k}}$ for $\mathsf{k}_{-s}^{e,\ell}$.
From \cite[Formula (3.8)]{Gerber2016}, 
it is straightforward that $\dot{\mathsf{k}}\circ\mathsf{k}=\id_{\Pi_s^\ell}$ and $\mathsf{k}\circ \dot{\mathsf{k}}=\id_{\Pi_{-s}^e}$.

We can extend $\mathsf{k}$ linearly to $\cF_{e,s}$, which endows it with the structure of a $\Ul$-module, 
by considering the natural action on $\cF_{\ell,-s}$ and composing with $\dot{\mathsf{k}}$. 
This yields an $\sll$-crystal structure on $\cF_{e,s}$.
More precisely, if we denote by $\tdf_j$, $j=0,\ldots,\ell-1$ the lowering $\sll$-crystal operators,
the action of $\tdf_j$ on an $\ell$-partition is defined by $\dot{\mathsf{k}}\circ\tilde{\dot{f}}_{j}\circ\mathsf{k}$, as indicated by the following diagram:
\begin{equation}\label{lr_diag}
 \begin{tikzcd}
 \Pi^\ell_s \arrow[dashrightarrow]{d}{} \arrow{r}{\mathsf{k}}
 &
 \Pi^e_{-s} \arrow{d}{\tilde{\dot{f}}_{j}} 
 \\
 \Pi^\ell_s 
 &
 \Pi^e_{-s} \arrow{l}{\dot{\mathsf{k}}}
 \end{tikzcd}
 \end{equation}

\begin{remark}\label{Ugyv}
\begin{enumerate}
 \item As explained in \cite[Section 7.1]{Gerber2016}, the map $\mathsf{k}$ is, up to a twist by conjugation,
categorified by \textit{Koszul duality} between the corresponding Cherednik categories $\cO$. This justifies the notation.
\item The level-rank duality $\mathsf{k}$ used in our paper is  
not the same as the one  used in Yvonne and Uglov's paper. However, our map can be recovered from  Uglov and Yvonne's ones by composing with the map  $|\bla,\bs\rangle \mapsto  |\bla^\mathrm{tr}_{\rev},-\bs_{rev}\rangle$, where for $\bla:=(\lambda^{(1)},\ldots,\lambda^{(\ell)})$ we have $\bla^\mathrm{tr}_{\rev}=((\lambda^{(\ell)})^\mathrm{tr},\ldots,(\lambda^{(1)})^\mathrm{tr})$ 
 and $\lambda^\mathrm{tr}$ is the transpose of $\lambda$. 
\end{enumerate}
\end{remark}

For $s\in \mathbb{Z}$, we denote 
$$A (s)=\left\{ (s_1,\ldots,s_\ell)\in \mathbb{Z}^\ell (s)\ |\ s_1\leq \ldots \leq s_\ell \leq s_1+e \right\}$$
and, in a dual fashion,
$$\dot{A} (s)=\left\{ (t_1,\ldots,t_e)\in \mathbb{Z}^e (s)\ |\ t_1\leq \ldots \leq t_e \leq t_1+\ell \right\}$$
Write $\dot{\bemp}=(\emp,\ldots,\emp)\in\Pi^e$. 
Note that for $\bs\in A(s)$, the set  $\ug_{e,\bs}$ has a  convenient non-recursive definition,  see \cite[Theorem 2.10]{FLOTW1999}.
By \cite[Formula (3.8)]{Gerber2016}, if $\bs\in A(s)$, then 
$\mathsf{k} |\bemp,\bs\rangle = |\dot{\bemp},\dbs\rangle$ for some $\dbs\in\dot{A}(-s)$.

Finally, there is an $\slinf$-crystal structure on $\Pi_s^\ell$ arising from the action of a Heisenberg algebra \cite{ShanVasserot}, \cite{Losev2015}, \cite{Gerber2016a}. 
Its connected components are all isomorphic to the branching graph of the symmetric group in characteristic $0$ and thus have vertices in bijection with $\Pi$. 
If $\bla_0$ is a highest weight vertex for the $\slinf$-crystal, then any $\ell$-partition in the same crystal component as $\bla_0$ is obtained as 
$\tilde{a}_\sigma(\bla_0)$ for a unique $\sigma\in\Pi$,
where $\ta_\sigma$ denotes the Heisenberg crystal operator associated to $\sigma$, see \cite{Losev2015} \cite{Gerber2016a}.

We will make repeated use of the following important theorem, proved in \cite[Theorems 6.17 and 6.19]{Gerber2016}, and its corollary.

\begin{theorem}\label{3crystals} \
\begin{enumerate}
 \item The three crystals pairwise commute.
 \item Every $|\bla,\bs\rangle\in\Pi^\ell_\bs$ decomposes as
 $$|\bla,\bs\rangle=\tdf_{j_r}\ldots \tdf_{j_1} \ta_{\sigma}   \tf_{i_p}\ldots \tf_{i_1}|\bemp,\br\rangle$$
 for some $\br \in A(s)$, $\sigma\in \Pi$, $p,r\in \mathbb{Z}_{\geq 0}$ and for some $i_p,\ldots,i_1\in \{0,1,\ldots,e-1\}$ and $j_r,\ldots,j_1\in \{0,1,\ldots,\ell-1\}$.
\end{enumerate}
\end{theorem}

\begin{corollary}\label{bijtriple}
The elements $\br,\si,p$ and $r$ of Theorem \ref{3crystals} are uniquely determined by $|\bla,\bs\rangle$. 
This yields a bijection
$$
\begin{array}{rrcl}
\beta : & \Pi^\ell_s & \lra &\ds \bigsqcup_{\br \in A (s)}  \ug_{e,\br} \times \Pi \times \ug_{\ell,\dbr} \\
& |\bla,\bs\rangle &\longmapsto & ( \tf_{i_p}\ldots \tf_{i_1}|\bemp,\br\rangle, \sigma, \tdf_{j_r}\ldots \tdf_{j_1}|\dot{\bemp},\dot{\br}\rangle).
\end{array}
$$
\end{corollary}

\begin{proof}
Let $\bla\in\Pi^\ell$. By Theorem \ref{3crystals} (2), there exist 
$\br \in A(s)$, $\sigma\in \Pi$, $p,r\in \mathbb{Z}_{\geq 0}$ and 
elements  $i_p,\ldots,i_1\in \{0,1,\ldots,e-1\}$ and $j_r,\ldots,j_1\in \{0,1,\ldots,\ell-1\}$ such that 
$$|\bla,\bs\rangle=\tdf_{j_r}\ldots \tdf_{j_1} \ta_{\sigma}   \tf_{i_p}\ldots \tf_{i_1}|\bemp,\br\rangle$$
Assume that we have:
$$\tdf_{j'_{r'}}\ldots \tdf_{j'_1} \ta_{\sigma'}   \tf_{i'_{p'}}\ldots \tf_{i'_1}|\bemp,\br'\rangle =\tdf_{j_r}\ldots \tdf_{j_1} \ta_{\sigma}   \tf_{i_p}\ldots \tf_{i_1}|\bemp,\br\rangle$$
for $\br' \in A(s)$, $\sigma'\in \Pi$, $p',r'\in \mathbb{Z}_{\geq 0}$ and 
indices  $i'_{p'},\ldots,i'_1\in \{0,1,\ldots,e-1\}$ and $j'_{r'},\ldots,j'_1\in \{0,1,\ldots,\ell-1\}$.
Then the elements  $|\bmu',{\bf t} '\rangle:= \tdf_{j'_{r'}}\ldots \tdf_{j'_1} \ta_{\sigma'}    |\bemp,\br' \rangle$ and 
$|\bmu,{\bf t} \rangle:=\tdf_{j_r}\ldots \tdf_{j_1} \ta_{\sigma}  |\bemptyset,\br\rangle$ are both highest weight vertices in the $\sle$-crystal.
As we have  $\tf_{i'_{p'}}\ldots \tf_{i'_1} |\bmu',{\bf t} '\rangle =\tf_{i_{p}}\ldots \tf_{i_1} |\bmu,{\bf t} \rangle$, these two elements are in the same connected component of 
the $\sle$-crystal so they must be equal.  From this equality, we deduce 
in the same way  that the two $\sle$-highest weight vertices must be equal:  
$\ta_{\sigma'}    |\bemp,\br' \rangle= \ta_{\sigma}    |\bemp,\br \rangle$.  
By the description of the $\slinf$-crystal operators \cite{Losev2015} \cite{Gerber2016}, we obtain $\sigma=\sigma'$ and $\br'=\br$.
We deduce that  $ \tdf_{j'_{r'}}\ldots \tdf_{j'_1}    |\dot{\bemp},\dbr \rangle=\tdf_{j_r}\ldots \tdf_{j_1}  |\dot{\bemp},\dbr\rangle$,
where $|\dot{\bemp},\dbr\rangle=\mathsf{k}|\bemp,\br\rangle$.
In particular, we have $r'=r$.
Using the same argument but exchanging the role of $e$ and $\ell$, we also get 
$\tf_{i'_{p'}}\ldots \tf_{i'_1} |\bemp,\br \rangle =\tf_{i_{p}}\ldots \tf_{i_1} |\bemp,\br \rangle$ and $p'=p$.
This proves uniqueness, and therefore $\be$ is well-defined.

For $\br \in A (s)$ and $(\bnu, \bpi)\in  \ug_{e,\br} \times \ug_{\ell,\dbr} $,
by \cite[Theorem 2.10]{FLOTW1999}, there exist indices $i_1,\ldots,i_p\in \{0,1,\ldots,e-1\}$ and $j_1,\ldots,j_r\in \{0,1,\ldots,\ell-1\}$
such that 
$$\tf_{i_p}\ldots \tf_{i_1}|\bemp,\br\rangle=|\bnu,\br\rangle \text{ and }\tdf_{j_r}\ldots \tdf_{j_1}|\dot{\bemp},\dot{\br}\rangle =|\bpi,\dbr\rangle.$$
Now, the map $\de : \bigsqcup_{\br \in A (s)}  \ug_{e,\br} \times \Pi \times \ug_{\ell,\dbr} \to \Pi^\ell_s,
(\bnu,\sigma,\bpi)\mapsto \tdf_{j_r}\ldots \tdf_{j_1} \ta_\si\tf_{i_p}\ldots \tf_{i_1}|\bemp,\br\rangle$
is well-defined, since it does not depend on the choice of the indices and the 3 crystals commute.
It is straightforward that $\be$ and $\de$ are inverse to each other, which concludes the proof.
\end{proof}

\begin{example}\label{example_triple}
Take $s=1$, $e=3$, $\ell=4$, $\bs=(-3,2,1,1)$ and $\bla=(\emp, 3.2^2,\emp,3)$.
Then
\begin{align*}
\beta ( |\bla,\bs\rangle) & = ( \,\, |(\emp,\emp,\emp,3),(-1,0,0,2)\rangle \,\, , \,\, (2) \,\, , \,\, |(2^2,2.1,\emp),(-1,-1,1)\rangle \,\,) \\
& =( \,\, \tf_1\tf_0\tf_2 \, |(\emp,\emp,\emp,\emp),(-1,0,0,2)\rangle \,\, , \,\, (2) \,\, , \,\, \tdf_0\tdf_2\tdf_3\tdf_3\tdf_0\tdf_2\tdf_3\, |(\emp,\emp,\emp),(-1,-1,1)\rangle \,\, ).
\end{align*}
\end{example}

\begin{remark}\label{rem_bij}
Note that for $\ell=1$, Corollary \ref{bijtriple} reduces to a very simple bijection.
Indeed, there is no $\sll$-crystal (and no level-rank duality) in this case,
and the bijection associates to any partition $\la$ a pair of partitions $(\rho, \si)$
determined by the ``euclidean division'' of $\la$ by $e$, as follows.
Given two partitions $\mu$ and $\mu'$, let $\mu \sqcup \mu'$ be the partition obtained by concatenating the two partitions and then reordering the parts to obtain a partition 
 (see for instance \cite[Section 3.1]{DudasJacon2018}). Then we can uniquely write
$$\la =(\sigma)^e \sqcup  \rho$$
where $\rho$ is an $e$-regular partition and $\si\in\Pi$.
\end{remark}

\begin{example}
Choose $e=3$ and $\la=(4^4.3^2.2.1^8)$. Then 
$\la = (4.1^2)^3 \sqcup (4.3^2.2.1^2)$.
\end{example}

\subsection{The generalized Mullineux map} \label{subsec_genmul}
In the following, we will need to go from one  indexation by $\ell$-partitions to the other by $e$-partitions using the map $\mathsf{k}$. 
We will use the  relationship between the weight spaces for the action of $\sle$ and the weight spaces for the action of $\sll$ of $\cF_{s,e}$.

We start by defining a map $\theta_{\ell,e,s}$ by setting 
$$\begin{array}{lccc}
\theta_{\ell,e,s} :& \mathbb{Q}^{\ell} (s) &\lra &\mathbb{Q}^{\ell} (e)\\
 & (s_1,\ldots, s_{\ell}) & \longmapsto & (e-s_1+s_{\ell},s_1-s_2,\ldots,s_{{\ell}-1}-s_{\ell}).
 \end{array}
$$
This is a bijection with inverse map 
$$\begin{array}{lccc}
\theta_{\ell,e,s}^{-1} :& \mathbb{Q}^{\ell} (e) &\lra &\mathbb{Q}^{\ell} (s)\\
 & (a_1,\ldots, a_{\ell}) & \longmapsto & (s_1,\ldots,s_{\ell}),
 \end{array}
$$
where we have for all $1\leq i\leq \ell$:
$$s_i=\frac{1}{\ell} (s-\sum_{1\leq j\leq \ell-1} j a_{j+1} )+\sum_{i+1\leq j\leq \ell} a_j.$$

\begin{lemma}\label{combij}
Keeping the above notation, assume that $\bs=\theta_{\ell,e,s}^{-1} (a_1,\ldots,a_{\ell})$ then we have $-\bs_\rev=\theta_{\ell,e,-s}^{-1} (a_1,a_\ell,\ldots,a_{2})$
\end{lemma}

\begin{proof}
Write $\bs=(s_1,\ldots,s_\ell)$ and ${\bf v}=\theta_{\ell,e,-s}^{-1} (a_1,a_\ell,\ldots,a_{2})$. On the one hand, we have for all $i=1,\ldots,\ell$
$$s_i=\frac{1}{\ell} (s-\sum_{1\leq j\leq \ell-1} j a_{j+1} )+\sum_{i+1\leq j\leq \ell} a_j,$$
and on the other hand:
$$\begin{array}{rcl}
v_{\ell-i+1}&=&\displaystyle{\frac{1}{\ell} (-s-\sum_{1\leq j\leq \ell-1} j a_{\ell-j+1} )+\sum_{\ell-i+2\leq j\leq \ell } a_{\ell-j+2}}\\
&=& \displaystyle{\frac{1}{\ell} (-s-\sum_{1\leq j\leq \ell-1} (\ell-k)  a_{k+1} )+\sum_{2\leq k\leq i } a_{k}}
\end{array}
$$
We obtain:
$$s_j+v_{\ell-j+1}=0$$
and the result follows. 
\end{proof}

We have the following result whose proof can be found in  \cite[Proposition 2.12]{Yvonne2007}, taking into account Remark \ref{Ugyv} (2).

\begin{proposition}\label{yv}
Let $\dbs=(\dot{s}_1,\ldots,\dot{s}_e)\in \mathbb{Z}^e (s)$ and let  $\dot{w}\in\dot{\cP}$ be a weight for $\cF_{\ell,\dbs}$. Then there exists a unique 
$\bs\in \mathbb{Z}^\ell (s)$ and a unique $w\in\cP$ such that $\mathsf{k}\left(\cF_{e,\bs}[w]\right)=\cF_{\ell,\dbs}[\dot{w}]$. 
If we write $\dot{w}=d\dot{\delta} +\sum_{0\leq i\leq \ell-1} a_i \dot{\Lambda}_{i-1}$ with $(a_1,\ldots, a_{\ell})\in \mathbb{Z}^\ell$  we have 
$$\bs=\theta_{\ell,e,s}^{-1} (a_1,a_{\ell},\ldots,a_2).$$ 
Moreover the associated weight is 
$$w=d\delta+\sum_{0\leq i\leq e-1} (\dot{s}_i-\dot{s}_{i+1}) \Lambda_i$$
where $\dot{s}_0=\ell+\dot{s}_e$.
\end{proposition}

We are now ready to prove the first main result of this paper.
Recall the generalized Mullineux map ${\tt\Phi}_{e,\bs}$ on Uglov $\ell$-partitions of \Cref{FJLmul}.
By level-rank duality, we have a dual Mullineux map ${\tt\Phi}_{\ell,\bt}$ for all $\bt\in\Z^e(-s)$
which acts on $e$-partitions.

\begin{theorem}\label{genmul} \ 
\begin{enumerate}
 \item There exists a unique bijection
$$\begin{array}{rccc}
\Phi:& \Pi^\ell_s &\lra& \Pi^\ell_{-s} \\
& |\bla,\bs \rangle & \longmapsto & |\bmu,-{\bs}_\rev  \rangle
\end{array}
$$
such that for all $0\leq i\leq e-1$, $\sigma \in \Pi$ and  $0\leq j\leq \ell-1$,
\begin{enumerate}
 \item $\Phi \circ \tf_i  = \tf_{-i} \circ \Phi$
\item $\Phi \circ \ta_{\sigma}  = \ta_{\sigma^\mathrm{tr}} \circ \Phi $
\item $\Phi \circ \tdf_j   = \tdf_{-j} \circ \Phi $.
\item $\Phi  (|\bemptyset,\bs \rangle)  =  |\bemptyset,-{\bs}_\rev \rangle$.
\end{enumerate}
\item Using the notation of Corollary \ref{bijtriple}, we have
$$\Phi = \beta^{-1} \circ ({\tt\Phi}_{e,\br}, (.)^\mathrm{tr}, {\tt\Phi}_{\ell,\dbr}) \circ \beta.$$
In other words, writing $|\bla,\bs\rangle=\tdf_{j_r}\ldots \tdf_{j_1} \ta_{\sigma}   \tf_{i_p}\ldots \tf_{i_1}|\bemp,\br\rangle$ with $\br\in A(s)$, 
we have $\Phi|\bla,\bs\rangle=\tdf_{-j_r}\ldots \tdf_{-j_1} \ta_{\sigma^\mathrm{tr}}   \tf_{-i_p}\ldots \tf_{-i_1}|\bemp,-\br_\rev\rangle$.
\item We have $|\bla|=|\bmu|$ if $|\bmu,-\bs_\rev\rangle=\Phi(|\bla,\bs\rangle)$.
\end{enumerate}
\end{theorem}
\begin{proof}
Let $\bla\in\Pi^\ell$, $\bs\in\Z^\ell(s)$ and 
write  $|\bla,\bs\rangle=\tdf_{j_r}\ldots \tdf_{j_r} \ta_{\sigma} \tf_{i_p}\ldots \tf_{i_1}  |\bemp, {\br} \rangle$
with $\br \in A(s)$ as in  Theorem \ref{3crystals}. If $\Phi$ satisfies the four assumptions of the Theorem, we have that:
$$\Phi(|\bla,\bs\rangle)=\tf_{-i_p}\ldots \tf_{-i_1} \ta_{\sigma^\mathrm{tr}} \tdf_{-j_r}\ldots \tdf_{-j_1}  |{\bemp},-{\br}_\rev\rangle$$
and this shows uniqueness.  Now, to prove $(1)$ and $(2)$, 
we need to show that there exists an $\ell$-partition $\bmu$ such that
$$|\bmu,-\bs_\rev\rangle=\tf_{-i_p}\ldots \tf_{-i_1} \ta_{\sigma^\mathrm{tr}} \tdf_{-j_r}\ldots \tdf_{-j_1}  |{\bemp},-{\br}_\rev\rangle$$
First, note that $\mathsf{k}|\bemp,-\br_\rev\rangle= |\dot{\bemp},-\dbr_\rev\rangle$ by \cite[Formula (3.8)]{Gerber2016}.
Consider the $e$-partition $\bla_2$ such that 
$$|\bla_2,\dot{\br}\rangle=\tdf_{j_r}\ldots \tdf_{j_1}   |\dot{\bemp},\dot{\br}\rangle$$
and the $e$-partition ${\bmu_2}$ such that 
$$|{\bmu_2},-\dot{\br}_\rev \rangle=\tdf_{-j_r}\ldots \tdf_{-j_1}  |\dot{\bemp},-\dot{\br}_\rev\rangle$$
defined thanks to \Cref{FJLmul}, i.e. 
$$|\bmu_2,-\dot{\br}_\rev \rangle={\tt\Phi}_{\ell,\dbr}(|\bla_2,\dot{\br}\rangle).$$
Let $|\bla_1,\bs \rangle=\dot{\mathsf{k}} (|\bla_2,\dot{\br}\rangle)$, so that
$|\bla,\bs\rangle= \tf_{i_p}\ldots \tf_{i_1}  \ta_\si| \bla_1, \dot{\bs} \rangle$.
 Let $\bv\in\Z^\ell(-s)$ and  $\mu_1$ be the $\ell$-partition  such that   
$|\bmu_1,\bv\rangle=\dot{\mathsf{k}} (|\bmu_2,-\dot{\br}_\rev\rangle)$.  
Let us show that $\bv=-\bs_\rev$.
By hypothesis, the $\sll$-weight $\dot{w}$ of $|\bla_2,\dot{\br}\rangle$ can be written
$$\dot{w}=d\dot{\delta}+\sum_{1\leq i\leq \ell}  \dot{\Lambda}_{r_i } -\sum_{1 \leq j\leq t} \alpha_{i_j} $$
and there exists a sequence of non negative integers $(a_1,\ldots,a_\ell)$ such that 
$$\dot{w}=d\dot{\delta}+\sum_{1\leq i\leq \ell} a_i \dot{\Lambda}_{i-1}.$$
By  Proposition \ref{yv}, with these notations, we have   $\bs=\theta_{\ell,e,s}^{-1} (a_1,a_\ell,\ldots a_2)$
By hypothesis, the $\sle$-weight $\dot{w}'$ of  $|\bmu_2,-\dot{\br}_\rev\rangle$ can be written
$$\dot{w}'=d'\dot{\delta}+\sum_{1\leq i\leq \ell}  \dot{\Lambda}_{-r_i } -\sum_{1 \leq j\leq t} \dot{\alpha}_{-i_j}$$
and thus we have 
$$\dot{w}'=d'\dot{\delta}+\sum_{1\leq i\leq \ell} a_i \dot{\Lambda}_{1-i} $$
We thus have $\bv=\theta_{\ell,e,-s}^{-1} (a_1,a_2, \ldots a_\ell)$. We conclude  that 
$\bv=-\bs_\rev$ using Lemma  \ref{combij}. 
Therefore, we can set $\bmu$ to be the $\ell$-partition such that $|\bmu,-\bs_\rev\rangle = \be^{-1} (\bmu_1,\si,\bmu_2)$,
see \Cref{bijtriple}, and this concludes the proof of $(1)$ and $(2)$.

It remains to prove $(3)$. To do this,  let us study the  $\sle$-weight of $|\bla_1,\bs\rangle$ and $|\bmu_1,-\bs_\rev \rangle$. 
Again by Prop \ref{yv}, the weight $w$ of $|\bla_1,\bs \rangle$  is 
 $$w=d\delta + (e-r_1+r_e) \Lambda_0+( r_1-r_2) \Lambda_1+\ldots +  ( r_{e-1}-r_e) \Lambda_{e-1} $$
 which can be written as
  $$w=d\delta +\sum_{1\leq i\leq l} \Lambda_{t_i}-\sum_{0\leq i\leq e-1} m_i \alpha_i $$
  for a sequence of non negative  integers $(m_i)_{i=0,\ldots,e-1}$. 
Note that , with this notation,  the number $N:=\sum_{0\leq i\leq e-1} m_i$ correspond to the size  of $\bla_1$ (see Section \ref{fock}).
Now again,  the weight $w'$ of $|\bmu_1,-\bs_\rev \rangle$  is 
$$w'=d\delta + (e-r_1+r_e) \Lambda_0+( r_{e-1}-r_e) \Lambda_1+\ldots +  ( r_{1}-r_2) \Lambda_{e-1} $$
which thus can be written as:
$$w'=d\delta +\sum_{1\leq i\leq l} \Lambda_{-t_i}-\sum_{0\leq i\leq e-1} (m_{-i}) \alpha_i. $$
We conclude that $|\bla_1|=|\bmu_1|$, that is, $N:=\sum_{0\leq i\leq e-1} m_i$. It follows that $|\bmu|=|\bla|=N+|\sigma| e+r$.
\end{proof}

We may write $\Phi(\bla)$ instead of $\Phi(|\bla,\bs \rangle)$
when the charge $\bs$ is understood.

\begin{example}
Take the same values as in Example \ref{example_triple}.
Denote $\br=(-1,0,0,2)$, so that $\dbr=(-1,-1,1)$.
Then we have 
\begin{align*}
\Phi ( |\bla,\bs\rangle) 
& = 
\beta^{-1} \, 
( \,\, 
{\tt\Phi}_{e,\br} ( \tf_1\tf_0\tf_2 \, |\bemp,\br\rangle) 
\,\, , \,\, 
(2)^\mathrm{tr}
\,\, , \,\, 
{\tt\Phi}_{\ell,\dbr} ( \tdf_0\tdf_2\tdf_3\tdf_3\tdf_0\tdf_2\tdf_3\, |\dot{\bemp},\dbr\rangle )
\,\, ) 
\\
& = 
\beta^{-1} \, 
( \,\, 
 \tf_2\tf_0\tf_1 \, |(\emp,\emp,\emp,\emp),(-2,0,0,1)\rangle 
\,\, , \,\, 
(1^2)
\,\, , \,\, 
\tdf_0\tdf_2\tdf_1\tdf_1\tdf_0\tdf_2\tdf_1\, |(\emp,\emp,\emp),(-1,1,1)\rangle 
\,\, )
\\
& = 
\beta^{-1} \, 
( \,\, 
  |(\emp,1,\emp,2),(-2,0,0,1)\rangle 
\,\, , \,\, 
(1^2)
\,\, , \,\, 
|(\emp,2^2,2.1),(-1,1,1)\rangle 
\,\, )
\\
&= |(1,2.1,\emp,2^3),(-1,-1,-2,3)\rangle.
\end{align*}
\end{example}

The following corollary shows that $\Phi$ generalizes the map ${\tt\Phi}_{e,\bs}$
of Section \ref{mulAK}.

\begin{corollary}\label{mul_coincide_on_uglovs}
Let $e\geq 2$, $\bs\in\Z^\ell(s)$, and $n\geq0$.
For all $\bla\in\ug_{e,\bs}(n)$, we have 
$$\Phi(\bla) = {\tt\Phi}_{e,\bs}(\bla).$$
\end{corollary}

\begin{proof}
Let $\bs\in\Z^\ell(s)$. By Property $(3)$ of Theorem \ref{genmul}, we have
$$\Phi(|\bemp, \bs\rangle )= |\bemp, -{\bs}_\rev \rangle.$$
Now if $\bla\in\ug_{e,\bs}(n)$ there exists a sequence of Kashiwara operators such that:
$$\tf_{i_p}\ldots \tf_{i_1}|\bemp, \bs\rangle=|\bla, \bs\rangle.$$
So we can use Property $(1)$ of Theorem  \ref{genmul} to see that 
$$\Phi(\tf_{i_p}\ldots \tf_{i_1} |\bemp, \bs\rangle )=\tf_{-i_p}\ldots \tf_{-i_1} |\bemp, -{\bs}_\rev \rangle.$$
By definition of ${\tt\Phi}_{e,\bs}$ in \Cref{FJLmul}, we get the result.  
\end{proof}

\subsection{More on crystal isomorphisms}\label{ciso}
Let $P_\ell:=\mathbb{Z}^\ell$ be the $\mathbb{Z}$-module with  standard basis $\{z_i\ |\ i=1,\ldots, \ell\}$.
For $k=1,\ldots,\ell-1$, we denote by $\sigma_k$ the transposition $(k,k+1)$ of  $\mathfrak{S}_\ell$. The extended affine symmetric group 
$\widehat{\mathfrak{S}}_\ell$ is the semidirect product  $P_\ell \rtimes \mathfrak{S}_\ell$ with  the relations  given by $\sigma_i z_j=z_j \sigma_i$ for $j\neq i,i+1$ and $\sigma_i z_i \sigma_i=z_{i+1}$ for $i=1,\ldots,\ell-1$ and $j=1,\ldots,\ell$.  
It  acts faithfully on $\mathbb{Z}^\ell$   as follows: for any ${{\bs}}=(s_{1},\ldots ,s_{\ell})\in 
\mathbb{Z}^{\ell}$: 
$$\begin{array}{rcll}
\sigma _{c}.{{\bs}}&=&(s_{1},\ldots ,s_{c-1},s_{c+1},s_{c},s_{c+2},\ldots ,s_{\ell})&\text{for }c=1,\ldots,\ell-1 \text{ and }\\
z_i.{{\bs}}&=&(s_{1},s_{2},\ldots,s_i+e,\ldots ,s_{\ell})&\text{for }i=1,\ldots,\ell.
\end{array}$$
If $\bs$ and $\bs'$ are in the same orbit modulo the action of $\widehat{\mathfrak{S}}_\ell$, then there is an $\widehat{\mathfrak{sl}_e}$-crystal isomorphism $\Psi_{\bs \to \bs'}$
between the Fock spaces  $\cF_{e,\bs}$ and $\cF_{e,\bs'}$, that is a map:
$$\Psi_{\bs \to \bs'}:\Pi^\ell_\bs \to\Pi^\ell_{\bs'} $$
such that :
\begin{itemize}
\item $|\bla,\bs\rangle$ is a highest weight vertex in $\cF_{e,\bs}$ if and only if $|\Psi_{\bs \to \bs'}(\bla),\bs ' \rangle$ is a highest weight vertex in $\cF_{e,\bs'}$,
\item For all $\bla \in \Pi^\ell$, we have $\Psi_{\bs \to \bs'} (\tf_i |\bla,\bs\rangle)=\tf_i\Psi_{\bs \to \bs'} ( |\bla,\bs\rangle)$
\end{itemize}
These crystal isomorphisms have been explicity described in  \cite{JaconLecouvey2008}.
Let us now come back to our situation. 
Assume that $\bs\in \mathbb{Z}^\ell$ and choose any ${\bs '} \in \mathbb{Z}^\ell$ in the orbit of $-\bs$ modulo the action of the  extended affine symmetric group. 
 This is in particular the case for $-{\bs}_\rev$. There is a $\widehat{\mathfrak{sl}_e}$-crystal isomorphism between the Fock spaces  $\cF_{e,-\bs}$ and $\cF_{e,\bs'}$. Composing this map with $\Phi$ thus gives an isomorphism 
   between $\cF_{e,\bs}$ and $\cF_{e,\bs'}$.

\section{Combinatorics of perverse equivalences for cyclotomic Cherednik category $\cO$}\label{chered}

In Section \ref{mulAK}, we studied the Mullineux involution in the context of representations of cyclotomic Hecke algebras.
In this section, we use the results of Section \ref{section_genmul} to study the Mullineux involution in the context of representations of cyclotomic rational Cherednik algebras. 
The goal is to realize the generalized Mullineux involution as the permutation of $\Pi^\ell$ induced by certain perverse equivalences.
We follow Losev's approach \cite{Losev2015a}, \cite{Losev2015}.

\subsection{Representations of cyclotomic rational Cherednik algebras}\label{reps of cyclo RCA}

We can deform $\C[x_1,\ldots,x_n,y_1,\ldots,y_n]\rtimes\C W_{\ell,n}$ to obtain an algebra called the cyclotomic rational Cherednik algebra \cite{EtingofGinzburg2002}.
This deformation depends on parameters $\kappa\in\Q^\times$ and $\bs=(s_1,\ldots,s_\ell)\in\Q^\ell$ (the charge). We denote it $\mathsf{H}_{\ka,\bs}(n)$.
The charge $\bs$ is identified with $\bs+\alpha(1,1,\ldots,1)$ for any scalar $\alpha$, thus the parameter space is $\ell$-dimensional.
As a $\C$-vector space, $\mathsf{H}_{\ka,\bs}(n)=\C[y_1,\ldots,y_n]\otimes\C[W_{\ell,n}]\otimes\C[x_1,\ldots,x_n]$. This makes it possible to define a category $\cO$ for $\mathsf{H}_{\ka,\bs}(n)$ as the full subcategory of $\mathsf{H}_{\ka,\bs}(n)$-mod consisting of finitely generated $\mathsf{H}_{\ka,\bs}(n)$-modules which are locally nilpotent for the action of $\C[y_1,\ldots,y_n]$ \cite{GGOR2003}.
Let $\cO_{\ka,\bs}(n)$ denote the category $\cO$ of $\mathsf{H}_{\ka,\bs}(n)$, and
$\cO_{\ka,\bs}=\bigoplus_{n\in\Z_{\geq 0}} \cO_{\ka,\bs}(n)$.
 $\cO_{\ka,\bs}(n)$ is a highest weight category whose simple objects 
are indexed by $\Irr_\C W_{\ell,n}$, the irreducible representations of the underlying group algebra $\C W_{\ell,n}$ \cite{GGOR2003}, thus by
$\Pi^\ell(n)$.
Furthermore, the highest weight structure of $\cO_{\ka,\bs}(n)$ depends on a partial order on $\Pi^\ell(n)$ which is determined by the parameter $(\kappa,\bs)$.

\subsection{Branching rules and crystals}
 In this subsection  we restrict to the case of \textit{integral parameters}, that is
$$  \ka = \pm 1/e \text{ for some }e\in\Z_{\geq 2} \text{\quad and \quad } \bs\in\Z^\ell.$$
The (complexified) Grothendieck group of $\cO_{\ka,\bs}$ is a level $\ell$ Fock space.
Recall from Section \ref{fock} that the parameters $e,\bs$ for the Fock space
come from its $\Ue$-module structure.
Shan \cite{Shan2011} has shown that if $\ka=1/e$, 
there is a notion of branching rule arising from Bezrukavnikov and Etingof's parabolic induction functors \cite{BezrukavnikovEtingof},
which categorifies the $\sle$-crystal of the Fock space $\cF_{e,\bs}$
\cite[Theorem 6.3]{Shan2011}.
Moreover, there is a categorical Heisenberg action on $\cO_{\kappa,\bs}$ giving rise to the $\slinf$-crystal on $\cF_{e,\bs}$ \cite{ShanVasserot},
\cite{Losev2015}. 
We have:
\begin{theorem}(Shan-Vasserot \cite[Proposition 5.18]{ShanVasserot})\label{crystalcusp}
The following are equivalent: 
\begin{enumerate}
\item $\bla$ is a highest weight vertex for both the $\sle$- and $\slinf$-crystals on $\cF_{e,\bs}$,
\item $L(\bla)$ is killed by the categorical Heisenberg and $\sle$ annihilation operators,
\item $L(\bla)$ is finite-dimensional.
\end{enumerate}
\end{theorem}
Thus finite-dimensional simples are labeled by the source vertices of the $\sle$- and $\slinf$-crystals.

\subsection{Perverse equivalences}\label{subsection:perverse} Perverse equivalences are a special kind of derived equivalence introduced by Chuang-Rouquier \cite{ChuangRouquier2008},\cite{ChuangRouquier2017} which are well-suited for combinatorial applications. 
Let $\cA,\cA'$ be abelian categories with finitely many simple objects and in which every object has finite length. Let $S,S'$ be the sets of isomorphism classes of simple objects of $\cA,\cA'$ respectively. Let $0\subset S_0\subset S_1\subset\ldots\subset S_r=S$ be a filtration of $S$ and $0\subset S_0'\subset S_1'\subset\ldots\subset S_r'=S'$ a filtration of $S'$.
Let $\cA_i\subset\cA,\cA_i'\subset \cA'$ be the Serre subcategories generated by $S_i,S_i'$ respectively and let $\pi:\{0,1,\ldots,r\}\rightarrow\Z$ be a function. 
\begin{definition}[Chuang-Rouquier]\label{def:perverse}
A derived equivalence $F:D^b(\cA)\rightarrow D^b(\cA')$ is \textit{perverse} if for all $i\geq 0$ and all $s\in S_i\setminus S_{i-1}$, the complex $F(s)$ satisfies:
\begin{enumerate}
 \item for $j\neq \pi(i)$ all composition factors of $H^j(F(s))$ are in $S_{i-1}'$,
 \item all the composition factors of $H^{\pi(i)}(F(s))$ are in $S_{i-1}'$ except for a unique one in $S_i'\setminus S_{i-1}'$.
 \end{enumerate}
 \end{definition}
 A perverse equivalence $F:D^b(\cA)\rightarrow D^b(\cA')$ therefore gives rise to a canonical bijection $\mathfrak{f}:S\rightarrow S'$ 
 sending $s\in S_i\setminus S_{i-1}$ to $\mathfrak{f}(s):=H^{\pi(i)}(F(s))\mod \cA'_{i-1}$. In the following sections \ref{subsection:wcfunctors} and \ref{subsection:Ringel}, we show how the generalized Mullineux involution arises from the perverse equivalences given by wall-crossing functors and  Ringel duality.

\subsection{Wall-crossing functors}\label{subsection:wcfunctors} Let us recall Losev's construction of wall-crossing functors in \cite{Losev2015}. These are derived equivalences between category $\cO_{?,\dagger}(n)$'s for parameters which differ by a perturbation of the partial order on $\Pi^\ell(n)$.
Denote by $\fc_\Z\subset \Q^\times \times \Q^\ell$ the $\ell$-dimensional lattice in the parameter space $\C^\ell$ consisting of those parameters $c'=(\ka',\bs')$ that have
\textit{integral difference} with a fixed parameter $c=(\ka,\bs)$, i.e. such that 
$\ka'-\ka\in\Z$ and $\ka' (s'_i-s_j ') -\ka (s_i-s_j)\in\Z$ for all $1\leq i<j \leq \ell$.
Note that replacing the $\ell$-tuple $\bs=(s_1,\ldots,s_\ell)$ with $\bs+(\alpha,\ldots,\alpha)$ for $\alpha\in\Z$ does not affect the definition of the corresponding rational Cherednik algebra, see the formulas in \cite[\S 3.3, Theorem 6.10]{ShanVasserot}.
The lattice $\fc_\Z$ is the shift by $c$ of the dual lattice to the lattice spanned by certain hyperplane elements, see \cite[Section 2.1.4, Section 2.7, Section 4.1.3]{Losev2015}. It is isomorphic to $\Z^\ell$.
There is a finite set of hyperplanes in $\C^\ell$ called  walls 
dividing
$\fc_\Z$ into open cones called chambers. These hyperplanes are defined as follows. For each $\bla=((\la_1^1,\la_{2}^1,\ldots),\ldots (\la_1^\ell,\la_{2}^\ell,\ldots))\in \Pi^{\ell} (n)$, let 
$$[\bla]:=\{ (a,b,j)\ |\ 1\leq a,\ 1\leq b\leq \lambda^j_a,\   1\le j \leq \ell\}$$
be the Young diagram of $\bla$. For each $(a,b,j)\in [\bla]$, we define
$$ co_c (\gamma)=\kappa \ell(b-a)+\ell h_j$$
where $h_j=\kappa s_j-\frac{j}{\ell}$
and:
$$c_{\bla}:=\sum_{\gamma \in [\bla]} co_c(\gamma) $$
called the $c$-function; the formula in our cyclotomic case was given in \cite{GordonLosev}. By definition, the walls are the hyperplanes 
$\Pi_{\bla,\bla'}$ given by $c_{\bla}=c_{\bla'}$. 
Among these walls, we will be interested in the so called 
 ``essential walls'' which are the following ones:
\begin{itemize}
 \item[(a)] the wall $\ka=0$ between chambers containing parameters $(\kappa,\bs)$ such that the denominator $e$ of $\kappa$ satisfies $2\leq e\leq n$,  
 in terms of the above definition, it is of the form $\Pi_{\bla,\bla'}$ for some $[\bla]=[\bmu]\sqcup\{\gamma\}$
  and $[\bla']=[\bmu]\sqcup\{\gamma'\}$, for some multipartition $\bmu$ and two boxes $\gamma=(a,b,j)$ and $\gamma'=(a',b',j')$ of $\bmu$ with $j=j'$. 
 \item[(b)] the walls $h_i-h_j=\ka m$ with $i\neq j$, $m\in\Z$ and $|m|<n$ 
 between chambers containing  parameters such that $s_i-s_j-m\in\kappa^{-1}\Z$. In terms of the above definition they are of the form $\Pi_{\bla,\bla'}$ where $[\bla]=[\bmu]\sqcup\{\gamma\}$
  and $[\bla']=[\bmu]\sqcup\{\gamma'\}$ for a multipartition $\bmu$  and two boxes $\gamma=(a,b,i)$ and $\gamma'=(a',b',j)$. 
\end{itemize}

Two categories whose parameters lie in the same chamber are equivalent as highest weight categories \cite[Proposition 2.8]{Losev2015}.
On the other hand, the bounded derived categories of $\cO_{c}(n)$ and $\cO_{c'}(n)$ are derived equivalent
when $c:=(\ka,\bs)$ is obtained from $c':=(\ka',\bs')$ by crossing a wall to an adjacent chamber.
 Two chambers are separated by the  wall $\Pi_{\bla,\bla'}$ if and only if the sign of $c_{\bla}-c_{\bla'}$ in a chamber is opposite to $c_{\bla}-c_{\bla'}$ in the adjacent one.
 The derived equivalences $\mathsf{WC}_{c\leftarrow c'}:D^b(\cO_c)\rightarrow D^b(\cO_{c'})$ are called wall-crossing functors. They are defined by taking the derived tensor product with a Harish-Chandra bimodule, see \cite[Section 2.8]{Losev2015} or \cite{Losev2017} for the construction of these functors.

Losev proved that $\mathsf{WC}_{c\leftarrow c'}:D^b(\cO_c)\rightarrow D^b(\cO_{c'})$ is a perverse equivalence with respect to the filtration of simple modules by their supports (the function $\pi$ then picks out the dimension of the support) \cite[Proposition 2.12]{Losev2015}. Therefore $\mathsf{WC}_{c\leftarrow c'}:D^b(\cO_c)\rightarrow D^b(\cO_{c'})$ induces a canonical bijection $\mathsf{wc}_{c\leftarrow c'}$ on $\Pi^\ell$, called the combinatorial wall-crossing. 
For walls of type (b), the combinatorial wall-crossings have been studied in \cite{Losev2015}, \cite{JaconLecouvey2018}. 
In particular, they are given by the crystal isomorphisms of Section \ref{ciso} for the appropriate parameters \cite[Theorem 11]{JaconLecouvey2018}. 
We are here interested in the walls of type (a). First we need to see for which types of parameters they are defined. 
We will denote by $\wc$ the combinatorial wall-crossing from $\kappa>0$ to $\kappa'<0$ corresponding to the wall of type (a).

For the next proposition, we will use the following notation. For  $c:=(\ka,\bs)$ and $\pi\in \mathfrak{S}_\ell$,  we say that $c$ is
 \textit{$\pi$-asymptotic}  if we have for all $i=1,\ldots,\ell$:
$$\left(s_{\pi (i)} - \frac{\pi (i)}{\kappa \ell}\right)  -\left(s_{\pi (i+1)}- \frac{\pi (i+1)}{\kappa \ell}\right) >n-1.$$
Note that this is slightly more general than Losev's definition \cite{Losev2015},
and that both agree for integral parameters.
We say that $c$ is \textit{asymptotic} if it is $\pi$-asymptotic for some $\pi$.
\begin{proposition}\label{prop_wc}
\begin{enumerate}
\item Let $c=(\kappa,\bs)$ and $c'=(\kappa',\bs')$ in $\mathfrak{c}_{\mathbb{Z}}$ such that $\kappa$ and $\kappa'$ have the same sign and such that 
 $c$ and $c'$ are both  $\pi$-asymptotic  for some $\pi\in \mathfrak{S}_\ell$. Then there are no essential  walls between $c$ and $c'$.
\item 
Assume that $c:=(\ka,\bs)$  and $c':=(\ka-1,\bs')$  are two parameters with integral difference, with $\kappa>0$ and $\kappa':=\kappa-1<0$ lying  in two different chambers separated by a unique wall of 
type (a) (and no wall of type (b)). Then there exists $\pi\in \mathfrak{S}_\ell$ such that 
$c$ is $\pi$-asymptotic and $c'$ is $\pi'$-asymptotic where $\pi ' \in\fS_\ell$ is defined by 
$\pi'(i) = \pi(\ell-i+1)$ for all $i=1,\ldots,\ell$.
 Conversely, if $c$ and $c'$ are as above, they are separated by a unique wall of type (a) and no wall of type (b). 
 \end{enumerate}
\end{proposition}
\begin{proof}
Let  $(i_1,i_2)\in \{1,\ldots,\ell\}^2$ be such that $i_1\neq i_2$ and assume that 
$$\left(s_{i_1} - \frac{i_1}{\kappa \ell}\right)  -\left(s_{i_2}- \frac{i_2}{\kappa \ell}\right) >n-1.$$
Let $\bla$ be an arbitrary $\ell$-partition of $n$ and let  $\gamma_1 =(a_1,b_1,i_1)$ and 
$\gamma_2 =(a_2,b_2,i_2)$ be two boxes of the Young diagram.  By \cite[Example 5.5.18]{GeckJacon2011}, we have:
$$b_1-a_1-(b_2-a_2)\geq 1-n.$$
Thus we deduce
$$b_1-a_1+s_{i_1} - \frac{i_1}{\kappa \ell} >b_2-a_2+s_{i_2} - \frac{i_2}{\kappa \ell}$$
which implies that  $co_c (\gamma_1)>co_c (\gamma_2)$. The same calculation holds for $c'$ and we get $co_{c'}(\gamma_1)>co_{c'}(\gamma_2)$. Thus the order on boxes induced by the $c$-function is the same if we choose it with respect to $c$ or to $c'$, and therefore there is no wall of type (b) between $c$ and $c'$. Thus we cannot have any walls of type $\Pi_{\bla,\bla'}$ between the two parameters. 
This directly implies $(1)$.

Now we prove $(2)$. Assume thus that $c:=(\ka,\bs)$  and $c':=(\ka-1,\bs')$  have integral difference, with $\kappa>0$ and $\kappa':=\kappa-1<0$ lying  in two different chambers separated by a unique wall of 
type (a) (and no wall of type (b)).  Let    $(i_1,i_2)\in \{1,\ldots,\ell\}^2$ be such that $i_1\neq i_2$. Consider the $\ell$-partition  $\bla$ of $n-1$ such that 
 $\lambda^k=\emptyset$ for all $k\in \{1,\ldots,\ell\}\setminus \{i_1\}$ and  $\lambda^{i_1}=(n-1)$. We then consider the $\ell$-partition 
  $\bmu$ of $n$ obtained from  $\bla$ by adding $\gamma_1:=(1,n,i_1)$ and the $\ell$-partition 
  $\bnu$ of $n$ obtained from  $\bla$ by adding $\gamma_2:=(1,1,i_2)$. 
   If $c:=(\ka,\bs)$ does not lie in an essential wall, we have two cases to consider:
   \begin{itemize}
   \item If $c_{{\bmu}}<c_{\bnu}$ then we have $co_c (\gamma_1)<co_c (\gamma_2)$. We thus obtain:
   $$\kappa \ell \left(n-1+s_{i_1}-\frac{i_1}{\kappa \ell}\right) < \kappa \ell \left(s_{i_2}-\frac{i_2}{\kappa \ell}\right)$$
   which implies that:
   $$\left(s_{i_2} - \frac{i_2}{\kappa \ell}\right)  -\left(s_{i_1}- \frac{i_1}{\kappa \ell}\right) >n-1.$$
   Now  consider the $\ell$-partition $\bla'$ of $n-1$ such that 
 ${\lambda'}^k=\emptyset$ for all $k\in \{1,\ldots,\ell\}\setminus \{i_1\}$ and  ${\lambda'}^{i_1}=(1^{n-1})$. We then consider the $\ell$-partition 
  $\bmu'$ of $n$ obtained from  $\bla$ by adding $\gamma_1:=(n,1,i_1)$ and the $\ell$-partition 
  $\bnu'$ of $n$ obtained from  $\bla$ by adding $\gamma_2:=(1,1,i_2)$. We must have:
    $$\kappa \ell \left(1-n+s_{i_1}-\frac{i_1}{\kappa \ell}\right) < \kappa \ell \left(s_{i_2}-\frac{i_2}{\kappa \ell}\right)$$
but now as no walls of type (b) must be crossed to go to $c'$, we must also have:
    $$\kappa' \ell \left(1-n+s_{i_1}'-\frac{i_1}{\kappa \ell}\right) < \kappa' \ell \left(s_{i_2}'-\frac{i_2}{\kappa \ell}\right).$$
This implies that
  $$\left(s_{i_1} ' - \frac{i_1}{\kappa' \ell}\right)  -\left(s_{i_2}'- \frac{i_2}{\kappa '\ell}\right) >n-1$$
because $\kappa'$ is negative. 
\item If $c_{\bmu}>c_{\bnu}$, we now have  $co_c (\gamma_1)>co_c (\gamma_2)$ and 
   $$\kappa \ell \left(n-1+s_{i_1}-\frac{i_1}{\kappa \ell}\right) > \kappa \ell \left(s_{i_2}-\frac{i_2}{\kappa \ell}\right).$$
This implies that 
  $$\left(s_{i_2}'  - \frac{i_2}{\kappa' \ell}\right)  -\left(s_{i_1}'- \frac{i_1}{\kappa '\ell}\right) >n-1$$
and by considering the same $\bla'$  as above we now obtain 
  $$\left(s_{i_1}  - \frac{i_1}{\kappa' \ell}\right)  -\left(s_{i_2}- \frac{i_2}{\kappa '\ell}\right) >n-1.$$
\end{itemize}
We can thus conclude that $c$ is $\pi$-asymptotic and $c'$ is $\pi'$-asymptotic  for a certain $\pi\in \mathfrak{S}_\ell$. Now let us  take such 
 a pair $(c,c')$ and let $(i_1,i_2)\in \{1,\ldots,\ell\}^2$ such  that  $i_1\neq i_2$ and 
$$\left(s_{i_1} - \frac{i_1}{\kappa \ell}\right)  -\left(s_{i_2}- \frac{i_2}{\kappa \ell}\right) >n-1$$
 so that 
$$\left(s_{i_2}' - \frac{i_2}{\kappa' \ell}\right)  -\left(s_{i_1}'- \frac{i_1}{\kappa' \ell}\right) >n-1.$$
 We have already seen that for any $\ell$-partition $\bla$, and $\gamma_1 =(a_1,b_1,i_1)$ and 
$\gamma_2 =(a_2,b_2,i_2)$  two boxes of the Young diagram, we have $co_c (\gamma_1)>co_c (\gamma_2)$. But, the hypothesis also shows that $co_{c'} (\gamma_1)>co_{c'} (\gamma_2)$. This implies that no wall of type (b) is crossed between $c$ and $c'$. If we now take an arbitrary $l$-partition of $n-1$ and 
 consider two addable boxes $\gamma_1$ and $\gamma_2$ in the same component such that 
$co_c(\gamma_1)<co_c(\gamma_2)$, then we have $co_{c'}(\gamma_1)>co_{c'}(\gamma_2)$. This implies that a wall of type $(a)$ is crossed. This concludes the proof. 
\end{proof}

Assume that ${\bf s}$ is  $\pi$-asymptotic  for some $\pi\in\fS_\ell$ and assume moreover that we are in the integral parameter case.
We denote $\bs_{\mathrm{opp}}:=\bs'$  an associated  multicharge  which is $\pi'$-asymptotic with respect to the notation of the proposition. 
This is a slight abuse of notation because $\bs'$ is  of course not unique in general. However, two $\pi$-asymptotic parameters lie in the same chamber.  Thus, the associated categories are equivalent as highest weight categories, and we can identify these two parameters. 
 Moreover, note that one can assume that $\bs'$ and $\bs$ lives in the same orbit modulo the action of the affine symmetric group. 
  Indeed, if we denote ${\bf s}:=(s_1,\ldots,s_\ell)$. One can choose  $(k_1,\ldots,k_\ell)\in \mathbb{Z}^\ell$ so that 
   ${\bf s}'':=(s_1+k_1 e,\ldots, s_\ell+k_\ell e)$ is $\pi'$-asymptotic. In this case,  $c''=(\kappa',\bs'')$ and $c$ are in $\mathfrak{c}_{\mathbb{Z}}$ and again $c''$ and $c'$ lie in the same chamber.

\begin{proposition}\cite[Proposition 5.6]{Losev2015}\label{wccomm1} \
\begin{enumerate}
\item The $\sle$-crystal commutes with $\wc$.
\item The $\slinf$-crystal commutes with $\wc$ up to taking the transpose, i.e. 
$\wc \circ \ta_\si = \ta_{\si^\mathrm{tr}} \circ \wc$ for all $\si\in\Pi$.
\end{enumerate}
\end{proposition}

For the next definition, if $\bla$ is an arbitrary $\ell$-partition and ${\bf s}$ an arbitrary multicharge we denote by   
 $(|\bla,\bs \rangle)^\mathrm{tr}:=|\bla^{\mathrm{tr}},-\bs\rangle$ where $\bla^{\mathrm{tr}}$ is  the $\ell$-partition 
 $((\lambda^1)^{\mathrm{tr}},\ldots,(\lambda^\ell)^\mathrm{tr})$
  and $-\bs:=(-s_1,\ldots,-s_\ell)$.  Now it follows from the definition of the crystal operators and \cite[Section 4.1.4]{Losev2015} that when changing $\kappa$ to $-\kappa$, $()^\mathrm{tr}$ commutes with the $\sle$-crystal up to changing $\tilde{f}_i$ to $\tilde{f}_{-i}$  (see \cite[Section 3.2.3]{JaconLecouvey2010}), and commutes with  the $\slinf$-crystal  up to taking the transpose (when $\kappa<0$, the $\slinf$-crystal adds horizontal $e$-strips instead of vertical $e$-strips, see \cite[Section 4.2.3]{Losev2015}).  
We deduce the following result.

\begin{theorem}\label{wcmul1}
Assume that we are in the integral case,
and that moreover $(\ka,\bs)$ is asymptotic. Assume that $\bla=(\lambda^1,\ldots,\lambda^\ell)$ is in the connected component of empty multipartition for the $\sle$-crystal and the $\slinf$-crystal.
 Then we have: 
 $$\Psi_{-{\bs}_\rev\to -{\bs}_{\mathrm{opp}}} \circ \Phi  (\bla)=\wc (\bla)^\mathrm{tr}$$
\end{theorem}\label{wcviapsi}
\begin{proof}
By hypothesis, there exist $\si\in\Pi$ and $(i_1,\ldots,i_p)\in (\mathbb{Z}/e\mathbb{Z})^p$ such that 
$$ \tf_{i_p}\ldots \tf_{i_1}\ta_{\sigma} |\bemp, \bs\rangle=|\bla, \bs\rangle$$
From Theorem \ref{genmul} $(1)$  together with the definition of the crystal isomorphism in Section \ref{ciso}, we have that:
 $$\begin{array}{rcl}
 \Psi_{-{\bs}_\rev\to -{\bs}_{\mathrm{opp}}} \circ \Phi  (\tf_{i_p}\ldots \tf_{i_1}\ta_{\sigma} |\bemp, \bs\rangle)&=& 
\tf_{-i_p}\ldots \tf_{-i_1} \Psi_{-{\bs}_\rev\to -{\bs}_{\mathrm{opp}}} \circ \Phi  (\ta_{\sigma} |\bemp, \bs\rangle)\\
&=&\tf_{-i_p}\ldots \tf_{-i_1}   \ta_{\sigma^\mathrm{{tr}}} |\bemp, -\bs_{\mathrm{opp}}\rangle 
\end{array}$$
 using the explicit formulae of the crystal isomorphism together with the explicit formula of the action of $\ta_\si$ \cite[Section 5]{Gerber2016a}. 
On the other hand, by Proposition \ref{wccomm1}, we also get:
 $$ \wc  (\tf_{i_p}\ldots \tf_{i_1}\ta_{\sigma} |\bemp, \bs\rangle)= 
  \tf_{i_p}\ldots \tf_{i_1}  \ta_{\sigma^\mathrm{tr}} |\bemp, \bs_{\mathrm{opp}}\rangle$$
  where the right-hand-side is computed with respect to $\kappa'<0$. The result then follows by sending $\kappa'$ to $-\kappa'$ (the latter is positive) and applying our operator $()^\mathrm{tr}$. 
\end{proof}

This Theorem is in fact a generalization of a result by  Losev in level 1  \cite[Corollary 5.7]{Losev2015}, which we recover below.
Using the notation of Remark \ref{rem_bij}, denote 
$$
\begin{array}{cccc}
M_e : &  \Pi & \lra & \Pi \\
 & \la=(\sigma^e) \sqcup \rho & \longmapsto & (\sigma^\mathrm{tr})^e \sqcup m_e(\rho)
\end{array}
$$

\begin{corollary}[Losev]\label{Lmul}
Suppose $\ell=1$. 
Then for all $\la\in\Pi$,
$$\wc(\la) = (M_e(\la))^\mathrm{tr}.$$
\end{corollary}

\begin{proof}
In the case $\ell=1$, the charge is irrelevant (see also Remark \ref{rem_mulFayers}), thus so is $\Psi_{-\bs_\rev\to\ -\bs_{\mathrm{opp}}}$.
By Remark \ref{rem_bij},
$\la=(\sigma^e) \sqcup \rho$ is the level $1$ analogue of the decomposition of Corollary \ref{bijtriple}
used to define $\Phi$, and we can identify $M_e(\la)$ with $(m_e(\rho), (\sigma^\mathrm{tr})^e)$.
Thus by Theorem \ref{genmul} (2), we have $\Phi(\la)=(m_e(\rho), (\sigma^\mathrm{tr})^e)=M_e(\la)$, and we conclude using Theorem \ref{wcmul1}.
\end{proof}

We are able to partially generalize the level $1$ statement to level $\ell$:
\begin{theorem}\label{wcmul2}
Assume that we are in the integral case,
and that moreover $(\ka,\bs)$ is asymptotic.
Let $k\in \{1,\ldots,\ell\}$  be such that 
$$s_k=\operatorname{max} \{s_i\ \, ;\, i=1,\ldots,\ell\}$$ 
Assume that $\bla=(\lambda^1,\ldots,\lambda^\ell)$ is such that  $\lambda^i$ is $e$-regular for all $i\in \{1,\ldots,\ell-1\} \setminus \{k\}$ and $\lambda^k$ is arbitrary. 
The combinatorial wall-crossing $\wc(\bla)$ is then given by the formula:
$$\wc(\bla)=(m_e (\lambda^1),\ldots,M_e(\lambda^k),\ldots,m_e (\lambda^\ell))^\mathrm{tr}.$$
\end{theorem}

\begin{proof}
Let us first assume that   $\lambda^j$ is $e$-regular for all $j=1,\ldots,\ell$. 
By \cite[Proposition 3.1]{Losev2015}, we know that the wall-crossing is independent of the choice of a Weil generic parameter. In our situation this means that we can assume that 
 for each $\bla \in \Pi^\ell (n)$, if two boxes have the same residues then they are in the same component. 
We thus have an action of a Kac-Moody algebra as a tensor product of $\ell$ copies of $\sle$, one for each component of the multipartition.  
By Proposition \ref{wccomm1}, the associated Kashiwara operators commute with $\wc$. Moreover, we know that $\wc$ sends the empty multipartition to the empty multipartition and that for each $e$-regular partition $\lambda$, there exists a sequence of Kashiwara operators sending $\emptyset$ to $\lambda$. The result follows. 

Next, assume that $\bla=(\lambda^1,\ldots,\lambda^\ell)$ is such that  $\lambda^j$ is $e$-regular for all $j\in \{1,\ldots,\ell\} \setminus \{k\}$ and $\lambda^k$ is arbitrary.  
 By our assumption on $k$ and the definition of the action of the Heisenberg crystal operators $\ta_\si$ in \cite[Proposition 5.3]{Losev2015} , we know that there exists $\sigma\in\Pi$ 
 and $\bmu'=(\mu^1 ,\ldots,\mu^\ell)$ such that $\mu^i=\lambda^i $ if $i\neq k$, $\mu^k $ is $e$-regular, and  $\ta_{\sigma} \bmu=\bla$. The result then follows from  Proposition \ref{wccomm1}.
\end{proof}

We obtain the following interesting corollary which was not immediate from the crystal graph perspective.
\begin{corollary}
Assume that we are in the integral case and that  $(\ka,\bs)$ is an asymptotic parameter and assume that  $\bla=(\lambda^1,\ldots,\lambda^\ell)$  is a highest weight vertex such that each $\lambda^j$ is $e$-regular. Then 
 $\wc(\bla)=(m_e (\lambda^1),\ldots,m_e (\lambda^\ell))^\mathrm{tr}$ is a highest weight vertex for the opposite asymptotic parameter. 
\end{corollary}

\begin{example}
Take $\ell=2$ and $s\in \mathbb{Z}$ such that $s>n-1$  so that ${\bf s}=(0,s)$ is an asymptotic $2$-charge. The bipartitions $(\lambda^1,\lambda^2)$ which are both 
highest weight vertex for the $\slinf$-crystal and the $\sle$-crystal are exactly the ones satisfying $\lambda^2=\emptyset$ and one of the following conditions:
\begin{itemize}
\item $\lambda^1=\mu^e$  for a partition $\mu$.
\item $\lambda^1$ has exactly one good removable box of residue $s \mod e$.
\end{itemize}
Let us consider the second case and assume that $\lambda^1$ is $e$-regular. Then our theorem asserts that $\wc(\bla)=(m_e(\lambda^1)^\mathrm{tr},\emptyset)$. 
 This is consistent with the fact that it must be  both a 
highest weight vertex for the $\slinf$-crystal and the $\sle$-crystal because $m_e(\lambda^1)^\mathrm{tr}$ has exactly one removable box of residue $-s \mod e.$

\end{example}

\subsection{Ringel duality for $\cO_{\kappa,\bs}$}\label{subsection:Ringel}

Consider $\cO_{\kappa,\bs}(n)$, the cyclotomic Cherednik category $\cO$ of the complex reflection group $G(\ell,1,n)$ for parameters $\kappa=r/e$ such that $e\geq 2$ and $\mathrm{gcd}(r,e)=1$, and $\bs\in\Q^\ell$ satisfying $\kappa e s_i\in\Z$ for $i=1,\ldots,\ell$.
The category $\cO_{\kappa,\bs}(n)$ is a highest weight category with standard objects $\Delta(\bla)$, costandard objects $\nabla(\bla),$ 
simple objects $L(\bla)$, etc, indexed by $\{\bla\in\Pi^\ell(n)\}$ \cite{GGOR2003}. By general theory \cite[Appendix]{Donkin} there is a corresponding Ringel duality functor 
$$\mathsf{R}: D^b(\cO_{\kappa,\bs}(n))\stackrel{\sim}{\longrightarrow} D^b(^\vee\cO_{\kappa,\bs}(n))$$
induced by the derived Hom functor with respect to a full tilting module, and which was realized explicitly in \cite{GGOR2003}. 
The category $^\vee\cO_{\kappa,\bs}(n)$ is called the Ringel dual of $\cO_{\kappa,\bs}(n)$. It is shown in \cite{GGOR2003} that $^\vee\cO_{\kappa,\bs}(n)$ is again a cyclotomic Cherednik category $\cO$. In particular, it holds that $^\vee\cO_{\kappa,\bs}(n)\ \simeq \cO_{\kappa',\bs'}(n)$
for some parameters $\kappa',\bs'$ of the same rank and level, so that the Grothendieck group of $\bigoplus\limits_n\cO_{\kappa',\bs'}(n)$ is again a level $\ell$ and rank $e$ Fock space. 
Losev proved that Ringel duality is perverse with respect to the filtration by support of simple modules \cite[Lemma 2.5]{Losev2017}. As discussed in Section \ref{subsection:perverse}, this means that we can pick out a unique composition factor $L(\bmu)$ in the homology of the complex $\mathsf{R}(L(\bla))$ such that $\bmu$ has maximal bidepth in the $\sle$ and $\slinf$-crystals. Set $\mathsf{r}(\bla)=\bmu$. For the rest of this section, we assume $n$ is fixed and write $\cO_{\kappa,\bs}$ as shorthand for $\cO_{\kappa,\bs}(n)$.

Note that in Section \ref{subsection:wcfunctors}, we have defined wall-crossing functors as crossing a single essential wall. However, there is a derived equivalence $\mathsf{WC}_{c'\leftarrow c}:D^b(\cO_c)\stackrel{\sim}{\rightarrow}D^b(\cO_{c'})$ for any $c'$ in the lattice $\mathfrak{c}_\Z$ of parameters having integral difference with $c$, defined by taking the derived tensor product with a Harish-Chandra bimodule \cite{Losev2017}. By abuse of language let us refer to such equivalence as a wall-crossing functor whenever $c$ and $c'$ do not lie in the same chamber.

We are now going to recall Losev's comparison of the Ringel duality with the wall-crossing functor to a maximally distant chamber. The first part of the following proposition appears in \cite{Losev2017} immediately after \cite[Theorem 6.1]{Losev2017}. We add some details to the proof by justifying the assumption of the existence of a parameter $c'$ in the same lattice as $c$ such that the order on category $\cO_{c'}$ is opposite to that on $\cO_c$.

\begin{lemma}\label{lem_sleRingel} Assume the parameter $c=(\kappa,\bs)$ is as in the assumptions of \cite[Theorem 4.1]{Losev2017}. The following statements hold. \begin{enumerate}

\item The Ringel duality $\mathsf{R}:D^b(\cO_{\kappa,\bs})\rightarrow D^b(^\vee\cO_{\kappa,\bs})$ can be realized by the (inverse) wall-crossing functor to the opposite chamber.

\item The combinatorial bijection $\mathsf{r}$ induced by Ringel duality commutes with the $\sle$-crystal.
\end{enumerate}
\end{lemma}

\begin{proof} 
(1) This is a special case of \cite[Theorem 6.1]{Losev2017} where it is observed that it follows from \cite[Lemma 2.5]{Losev2017} and \cite[Theorem 4.1]{Losev2017}. 
We justify the existence of an opposite chamber and its intersection with the $\Z$-lattice of parameters $\mathfrak{c}_\Z$ in the case $W=G(\ell,1,n)$.

The category $\cO_{\kappa,\bs}$ is a highest weight category with respect to the order on $\Irr \ G(\ell,1,n)=\Pi^\ell(n)$ given by the $c$-function $c_{\bla}$ (see Section \ref{subsection:wcfunctors}). It is a general fact about Ringel duality for highest weight categories that the Ringel dual category has the same poset with the opposite partial order \cite[Appendix]{Donkin}. Thus the $c$-order on $\Pi^\ell(n)$ with respect to the parameters $(\kappa',\bs')$ is opposite to the $c$-order on $\Pi^\ell(n)$ with respect to the parameters $(\kappa,\bs)$. Recall from \cite{Losev2015} that the $c$-order is constant on the open chambers given by the complement of the hyperplane arrangement of essential walls. 

We claim that we can find an example of a parameter $c'=(\kappa',\bs')$ with the property that $\cO_{\kappa',\bs'}$ realizes the Ringel dual of $\cO_{\kappa,\bs}$ and $(\kappa',\bs')$ has integral difference with $(\kappa,\bs)$. Recall the definitions of walls and chambers from Section \ref{subsection:wcfunctors}. 
First, assume that $\ka>0$, and let $t\in \mathbb{Z}_{>0}$ such that $t{\bs}\in \mathbb{Z}^\ell$
and $n\in \mathbb{Z}$ be such that $\kappa-n<0$, $t$ divides $n$ and $n/(l\kappa)\in \mathbb{Z}$.
Set $c':=(\kappa',{\bf s}')\in \mathbb{Q}_{<0}^{\times} \times \mathbb{Q}^\ell$ such that $\kappa':=\kappa-n$ and for all $j=1,\ldots,\ell$, we have
$s_j '=s_j +\frac{jn}{\ell\kappa \kappa'}$.
Then for all $j=1,\ldots,\ell$, we have
$$\begin{array}{rcl}
\kappa' s_j'-\kappa s_j&=& (\kappa-n) (s_j+\displaystyle\frac{jn}{\ell\kappa (\kappa-n)}) -\kappa s_j \\
 &=&-ns_j+\displaystyle\frac{jn}{\ell \kappa} \in \mathbb{Z}.
\end{array}$$
Since we have in addition $\kappa-\kappa'\in \mathbb{Z}$, we deduce that $(\kappa,{\bf s})$ and $(\kappa',{\bf s}')$ are in the same lattice.
Now, let $\gamma=(a,b,j)\in \mathbb{Z}_{>0}\times \mathbb{Z}_{\geq 0} \times \{1,\ldots,\ell\}$ be a box of an $\ell$-partition. We have $$co_c (\gamma)=\kappa \ell (b-a+s_j-\frac{j}{\ell\kappa})\text{ and }co_{c'} (\gamma)=\kappa' \ell (b-a+s_j'-\frac{j}{\ell\kappa'})$$
 and for $j=1,\ldots,\ell$:
 $$\begin{array}{rcl}
s_j'-\displaystyle \frac{j}{\ell\kappa'}-(s_j-\frac{j}{\ell\kappa}) &=&\displaystyle\frac{jn}{\ell\kappa \kappa'} -  \frac{j}{\ell\kappa'}  + \frac{j}{\ell\kappa}\\
 &=&\displaystyle\frac{jn-j\kappa+j(\kappa-n)}{\ell\kappa (\kappa-n)}\\
 &=&0
\end{array}$$
As $\kappa$ and $\kappa'$ have opposite sign, we conclude that $c$ and $c'$
induce opposite orders with respect to the $c$-function.\\

(2) The proof follows by exactly the same argument as the proof of \cite[Proposition 5.6(1)]{Losev2015}. Indeed, the assumption of \cite[Proposition 5.6(1)]{Losev2015} that the wall-crossing is through a single essential wall is only used in reference to \cite[Proposition 3.8]{Losev2015}, which states that wall-crossing functors across essential walls intertwine the parabolic restriction functors. 
The analogous statement that (inverse) Ringel duality intertwines parabolic induction and restriction functors is proved in \cite[Lemma 4.7]{Losev2017}.
\end{proof}

The symmetrical statement to Proposition \ref{lem_sleRingel}(2) for the $\slell$-crystal follows by level-rank duality which switches the roles of $e$ and $\ell$ and commutes with Ringel duality.
For this, we need to be in the integral parameter case again.

\begin{lemma}\label{lem_sllRingel}
The map $\mathsf{r}$ commutes with the $\slell$-crystal.
\end{lemma}
\begin{proof}
 By \cite[Theorem 7.4]{RSVV} and \cite{Webster2017}, the category $\cO_{\kappa,\bs}$ is standard Koszul, which by \cite{Mazorchuk} implies that Ringel duality commutes with the Koszul duality. The Koszul duality $\mathsf{K}$ lifts the level-rank duality $\mathsf{k}$, by which the $\slell$-crystal is defined, to the categorical level \cite[Theorem 7.4]{RSVV}. Since we need to compare Ringel duality for level $e$ with Ringel duality for level $\ell$, write $\mathsf{R}$ for the former and $\dot{\mathsf{R}}$ for the latter; likewise, write ${\mathsf{r}}$ for the former and 
 $\dot{\mathsf{r}}$ for the latter on the level of Grothendieck groups. 
 We have
$\dot{\mathsf{R}}\mathsf{K}=\mathsf{K}\mathsf{R}$ and thus
 $\dot{\mathsf{r}}\mathsf{k}=\mathsf{k}\mathsf{r}$ and $\mathsf{r}\dot{\mathsf{k}}=\dot{\mathsf{k}}\dot{\mathsf{r}}$.
 Let $\tilde{\dot{f}}_j$ be an $\slell$ crystal operator. Recall from Diagram \ref{lr_diag} that the action of $\tilde{\dot{f}}_j$ in level $\ell$ and rank $e$ is defined by 
 $\dot{\mathsf{k}}\tilde{\dot{f}}_j\mathsf{k}$.
By the preceding lemma, in level $e$ and rank $\ell$ we have  $\tilde{\dot{f}}_j\dot{\mathsf{r}}=\dot{\mathsf{r}}\tilde{\dot{f}}_j$. Therefore, 
$\mathsf{r}(\dot{\mathsf{k}}\tilde{\dot{f}}_j\mathsf{k})=
\dot{\mathsf{k}}\dot{\mathsf{r}}\tilde{\dot{f}}_j\mathsf{k}=
\dot{\mathsf{k}}\tilde{\dot{f}}_j \dot{\mathsf{r}} \mathsf{k}=
(\dot{\mathsf{k}}\tilde{\dot{f}}_j\mathsf{k})\mathsf{r}.$
\end{proof}

\subsection{The combinatorial Ringel duality} 
For this section, we take again parameters to be integral.
The category $\cO_{\kappa,\bs}$ is equivalent as a highest weight category to the category $\cS_{\ka,\bs}$ 
of finite-dimensional modules over the diagrammatic Cherednik algebra with Uglov weighting defined by Webster \cite[Theorem 4.8]{Webster2017}.
Webster proved that for the category $\cS_{\ka,\bs}$, taking the Ringel dual corresponds to sending $\kappa$ to $-\kappa$ and keeping the charge $\bs$ fixed \cite[Corollary 5.11]{Webster2017}. 
Moreover, there is a natural isomorphism of diagram algebras
(which is given by replacing the label $i\in\Z/e\Z$ with $-i$ on black strands and the label $j\in\Z/\ell\Z$ with $-j$ 
on red strands, see \cite[Proposition 4.5]{Webster2017}). 
This isomorphism induces an equivalence $$\ast:\cS_{\ka,\bs} \rightarrow\cS_{-\ka,-\bs_\rev}$$ which evidently satisfies 
$\tilde{f}_{-i}\circ\ast=\ast\circ \tilde{f}_i$ and $\tilde{\dot{f}}_{-j}\circ\ast=\ast\circ\tilde{\dot{f}}_j$. Composing $*$ with Ringel duality for the diagrammatic Cherednik algebra then gives an equivalence (via identification of $\cO_{\kappa,\bs}$ with $\cS_{\kappa,\bs}$) $$\ast\circ \mathsf{R}:D^b(\cO_{\ka,\bs})\simeq D^b(\cS_{\ka,\bs})\stackrel{\sim}{\longrightarrow} D^b(\cS_{\ka,-\bs_\rev})\simeq D^b(\cO_{\ka,-\bs})$$
which is perverse, being the composition of abelian equivalences with a perverse equivalence \cite{ChuangRouquier2017}.

Set $\mathsf{D}=\ast\circ \mathsf{R}:D^b(\cO_{\kappa,\bs})\stackrel{\sim}{\rightarrow}D^b(\cO_{\kappa,-\bs_\rev})$ (c.f. \cite[Proposition 4.10]{GGOR2003}).
On the level of Grothendieck groups, 
$\mathsf{D}$ yields an involutive isomorphism of Fock spaces
$$\mathsf{d} : \cF_{e,\bs} \lra \cF_{e,-\bs_\rev}.$$
We call $\mathsf{d}$ the \textit{combinatorial Ringel duality}.

\begin{corollary}\label{twistedchapeauchapelle} For all $i=0,\ldots,e-1$ and all $j=0,\ldots,\ell-1$ we have
$$\mathsf{d}\tilde{f}_i=\tilde{f}_{-i}\mathsf{d}\qquad\hbox{ and }\qquad \mathsf{d}\tilde{\dot{f}}_j=\tilde{\dot{f}}_{-j}\mathsf{d} .$$
\end{corollary}

\begin{proof}
This is straightforward from Lemmas \ref{lem_sleRingel} and \ref{lem_sllRingel}, because $\mathsf{d}=\ast\circ\mathsf{r}$ and 
$\ast$ verifies $\tilde{f}_{-i}\circ\ast=\ast\circ \tilde{f}_i$ and $\tilde{\dot{f}}_{-j}\circ\ast=\ast\circ\tilde{\dot{f}}_j$, as explained above.
\end{proof}

\begin{lemma}\label{twistedslinf} For all $\sigma\in\Pi$, we have $\mathsf{d}\tilde{a}_\sigma=\tilde{a}_{\sigma^\mathrm{tr}}\mathsf{d}$.
\end{lemma}
\begin{proof} Since $\mathsf{D}$ is perverse with respect to supports of simple modules, i.e. cuspidal depth, $\mathsf{d}$ is an involutive isomorphism of $\slinf$-crystals $\cF_{e,\bs}\rightarrow\cF_{e,-\bs_\rev}$. Thus either $\mathsf{d}\circ\tilde{a}_\sigma=\tilde{a}_\sigma\circ\mathsf{d}$ or $\mathsf{d}\circ\tilde{a}_\sigma=\tilde{a}_{\sigma^\mathrm{tr}}\circ\mathsf{d}$.
By Corollary \ref{twistedchapeauchapelle}, $\mathsf{d}$ commutes with the two kinds of Kashiwara crystal operators up to a change of sign in the indices.
Therefore, by Theorem \ref{3crystals} (2), it is enough to prove that the two maps agree on multipartitions of the form $|\bemp,\bs\rangle$.
Further, by \cite{JaconLecouvey2018}, $|\bemp,\bs\rangle$ can be obtained by applying a sequence of combinatorial wall-crossings
of type (b) to $|\bemp,\bs'\rangle$ for some asymptotic parameter $(\ka,\bs')$.
By \cite[Proposition 5.6 (1),(2)]{Losev2015}, the combinatorial wall-crossings of type (b) commute with the three kinds of crystal operators.
Therefore, we can reduce the proof to the case where $(\ka,\bs)$ is asymptotic. 
Now, we have already seen that in this case, $\tilde{a}_\si$ acts only on one component of the empty multipartition,
so we can reduce the proof to the case $\ell=1$.

 For the rest of the proof, consider the functors $\mathsf{D}$ and $\ast\circ\WC$ when $\ell=1$ so that we are considering functors on the category $\cO_\kappa(\mathfrak{S}_n)$ (as $G(1,1,n)=\mathfrak{S}_n$). 
By \cite[Corollary 5.13]{Rouquier2008}, if $r>0$ is coprime to $e$ then $\cO_{-r/e}(\mathfrak{S}_n)\simeq \cO_{-1/e}(\mathfrak{S}_n)$ are equivalent as highest weight categories. Moreover, the equivalence sends $\Delta(\lambda)$ to $\Delta(\lambda)$, as follows from \cite[Corollary 6.10]{GGOR2003} together with the fact that the Hecke algebra at a primitive $e$'th root of unity depends only on $e$, not on the root of unity \cite{BrundanKleshchev}. 
Thus, though abusing notation, the composition $\ast\circ\WC$ is well-defined.

We claim that for $\ell=1$, the combinatorial Ringel duality $\mathsf{d}$ coincides with the transpose of the combinatorial wall-crossing. If $\kappa>0$ and $\kappa'<0$ and $\WC:D^b(\cO_\kappa(\mathfrak{S}_n))\stackrel{\sim}{\rightarrow} D^b(\cO_{\kappa'}(\mathfrak{S}_n))$ is the wall-crossing functor across the (unique) essential wall defined by the equation $\kappa=0$, then the combinatorial wall-crossing $\wc$ is given by $\wc(\lambda)=(M_e(\lambda))^{\mathrm{tr}}$ for all $\lambda\in\Pi(n)$ \cite[Corollary 5.7]{Losev2015}.
The functors $\mathsf{R}$ and $\WC$ are both perverse equivalences from $\cO_\kappa(\mathfrak{S}_n)$ with the same perversity function, namely the cuspidal depth of $L(\lambda)$. Thus $\mathsf{D}=\ast\circ\mathsf{R}$ and $\ast\circ\WC$ are perverse self-equivalences of $\cO_\kappa(\mathfrak{S}_n)$ with the same perversity function \cite[Lemma 2.5, Theorem 6.1]{Losev2017}. By \cite[Proposition 4.17]{ChuangRouquier2017}, if $\mathsf{F}:D^b(\cC)\stackrel{\sim}{\rightarrow}D^b(\cD)$ and $\mathsf{G}: D^b(\cC)\stackrel{\sim}{\rightarrow}D^b(\cE)$ are two perverse equivalences with the same perversity function, then $\mathsf{G}\circ\mathsf{F}^{-1}:\cD\stackrel{\sim}{\rightarrow} \cE$ is an equivalence of abelian categories. In particular, this implies that $\mathsf{R}(\cO_\kappa(\mathfrak{S}_n))\simeq \WC(\cO_\kappa(\mathfrak{S}_n))$.  Set $\mathsf{A}:=\left(\ast\circ\WC \right)\circ\mathsf{D}^{-1}:\cO_{\kappa}(\mathfrak{S}_n)\stackrel{\sim}{\rightarrow}\cO_{\kappa}(\mathfrak{S}_n)$. On $K_0$, it holds that $[\mathsf{D}(\Delta(\lambda))]=[\left(\ast\circ\mathsf{R}\right)(\Delta(\lambda))]=[\left(\ast\circ\WC\right)(\Delta(\lambda))]=[\Delta(\lambda^{\mathrm{tr}})]$ for any $\lambda\in\Pi(n)$ by \cite[Proposition 3.3, Corollary 4.8, Proposition 4.10]{GGOR2003} and \cite[Proposition 3.2]{Losev2015}. Thus $[\mathsf{A}(\Delta(\lambda))]=[\Delta(\lambda)]$ for all $\lambda\in\Pi(n)$. Moreover, it follows from the definitions of $\mathsf{R}$ and $\WC$ that $\mathsf{A}$ sends standardly filtered objects to standardly filtered objects. Therefore $\mathsf{A}(\Delta(\lambda))=\Delta(\lambda)$, from which it follows that $\mathsf{A}(L(\lambda))=L(\lambda)$. We conclude that $\mathsf{d}=(-)^{\mathrm{tr}}\circ\wc=M_e$. This concludes the proof.
\end{proof}

Now recall the generalized Mullineux involution $\Phi:\cF_{e,\bs}\rightarrow\cF_{e,-\bs_\rev}$ from Theorem \ref{genmul}. 
\begin{theorem}\label{ringelmul}
We have $\mathsf{d}=\Phi.$
\end{theorem}

\begin{proof}
Corollary \ref{twistedchapeauchapelle} and Lemma \ref{twistedslinf} combined imply that $\mathsf{d}$ and $\Phi$ satisfy the same commutation relations with the operators for the $\slinf$-, $\sle$-, and $\slell$-crystals.
Moreover, as seen in the proof of \Cref{twistedslinf}, $\mathsf{d}$ maps $|\bemp,\bs\rangle$ to $|\bemp,-\bs_\rev\rangle$.
By Theorem \ref{genmul} (1), $\Phi$ is the unique involution with these properties, so $\mathsf{d}=\Phi$.
\end{proof}

We recall for the reader the fact that the finite-dimensional modules of a rational Cherednik algebra $H_c(W)$ coincide with the cuspidal modules, i.e. those sent to $0$ by every parabolic restriction functor with respect to a proper parabolic subgroup of $W$  \cite{BezrukavnikovEtingof}. 
\begin{corollary}\label{level2}
Suppose $L(\bla)$ is a finite-dimensional irreducible representation of the rational Cherednik algebra of type $B_n=G(2,1,n)$ 
for parameters corresponding to Fock space charge $\bs=(s_1,s_2)\in\Z^2$. 
Then $\Phi(\bla)=\bla$, and $\mathsf{D}\left(L(\bla)\right)=L(\bla),$ where we identify $L(\bla)$ with the complex concentrated in degree $0$.
\end{corollary}
\begin{proof}
First, we make some remarks which hold for arbitrary $\ell$. By Theorem \ref{crystalcusp}, $L(\bla)$ is finite-dimensional if and only if $|\bla,\bs\rangle$ has depth $0$ in both the $\slinf$- and $\sle$-crystals. Recall that $\mathsf{D}$ is perverse with respect to filtration by the depth in the $\sle$- and $\slinf$-crystals, i.e. cuspidal depth. 
Since $\mathsf{D}$ is not just a derived equivalence but a perverse equivalence, by Definition \ref{def:perverse} it holds that $\mathsf{D}$ induces an abelian equivalence on each associated graded layer of $\cO_{\kappa,\bs}(n)$ with respect to the filtration by supports of simple modules. In particular, $\mathsf{D}$ restricts to an equivalence of abelian categories on the bottom filtration layer of $\cO_{\kappa,\bs}(n)$, and this subcategory coincides with the subcategory of cuspidal, that is finite-dimensional, modules. Thus if $\bla$ has depth $0$ in both the $\sle$- and $\slinf$-crystals, $\mathsf{D}$ sends $L(\bla)$ to some $L(\bmu)$ where $\bmu$ again has depth $0$ in both the $\sle$- and $\slinf$-crystals. Here,   we identify $L(\bla)$ with a complex concentrated in degree $0$, and likewise for $L(\bmu)$.   In this case, $\mathsf{D}(L(\bla))=L(\Phi(\bla))$ by Theorem \ref{ringelmul}.

Now we specify to the case $\ell=2$. When $\ell=2$, the map $\Phi$ fixes the $\slell$-crystal operators $\tdf_j$ since $j=-j\mod 2$. 
Moreover, $-\bs_\rev=(-s_2,-s_1) = (s_1,s_2) - (s_2+s_1,s_2+s_1)$, i.e. $-\bs_\rev$ is just an integer shift of $\bs$. Recall the remark at the beginning of \S\ref{reps of cyclo RCA} that charges differing by an integer shift define the same rational Cherednik algebra. Thus we may identify the Cherednik algebras defined by $(\kappa,\bs)$ and $(\kappa,-\bs_\rev)$, and hence we may identify their categories $\cO$.
Now, it is straightforward to see that the action of $\tdf_j$ on an element of $\Pi^\ell_\bs$ is not affected by an integer shift of the charge,
more precisely
$$\tdf_j|\bnu,\bt\rangle =: |\bnu',\bt'\rangle \,\, \Ra \,\,  \tdf_j|\bnu,\bt+k\rangle = |\bnu',\bt'+k\rangle$$
for all $\bnu\in\Pi^\ell$, $\bt\in\Z^\ell(s)$ and $k\in\Z$.
In particular, we get $\Phi(|\bla,\bs\rangle)=|\bla,-\bs_\rev\rangle$,
which proves the claim.
\end{proof}
\begin{remark}
Corollary \ref{level2} implies that Ringel duality can be seen as a self-equivalence of $\cO_{e,\bs}(B_n)$ which fixes cuspidal (i.e. finite-dimensional) modules,
and we recover \cite[Remark 4.12]{GGOR2003}.
\end{remark}

\begin{example}
Take $e=3$, $\bs=(-1,3)$ and $\bla=(3.3,\emptyset)$.
Then 
\begin{align*}
\beta (|\bla,\bs\rangle) & = 
(  \,\, |(\emp,\emp),(0,2)\rangle\,\, , \,\, \emptyset \,\, , \,\, |(2,2,1),(-1,-1,0)\rangle \,\, ) \\
& = ( \,\, |(\emp,\emp),(0,2)\rangle\,\, , \,\, \emptyset \,\, , \,\, \tdf_0\tdf_0\tdf_0\tdf_1\tdf_1 \, |(\emp,\emp,\emp),(-1,-1,0)\rangle \,\, ),
\end{align*}
so that $L(\bla)$ is finite-dimensional in $\cO_{1/e,\bs}(6)$.
We have 
\begin{align*}
\Phi(|\bla,\bs\rangle) & = \beta^{-1} (  \,\, |(\emp,\emp),(-2,0)\rangle\,\, , \,\, \emptyset \,\, , \,\, \tdf_0\tdf_0\tdf_0\tdf_1\tdf_1 \, |(\emp,\emp,\emp),(0,1,1)\rangle \,\, ) \\
& = \beta^{-1} (  \,\, |(\emp,\emp),(-2,0)\rangle\,\, , \,\, \emptyset \,\, , \,\, |(1,2,2),(0,1,1)\rangle \,\, ) \\
& = |(3.3,\emp),(-3,1)\rangle \\
& = |\bla,-\bs_\rev\rangle.
\end{align*}
\end{example}

\section{Other crystal structures and generalization of Mullineux involutions}\label{perspectives}

We now explain a possible generalization of the Mullineux involution defined in \cite{DudasJacon2018} that uses the results of Section \ref{section_genmul}.  
We already know that  we have three different crystal structures on level $\ell$ Fock spaces which are pairwise commuting,
namely the $\sle$-, $\slinf$- and $\sll$-crystal.
We now slightly generalize this by introducing an additional parameter $d\in \Z_{\geq0}$ 
(which plays the role of the parameter $\ell$ in \cite{DudasJacon2018} in level one - that is type $A$ situation).
 To do this, recall that if $\lambda\in \Pi$, we can uniquely  write the decomposition:
 $$\lambda=\lambda_{(0)}+\lambda_{(1)}^d+\ldots +\lambda_{(n)}^{d^n}$$ 
 for $n\in \Z_{\geq0}$ and where each $\lambda_{(i)}$ is $d$-regular.
 
 Let $j\in \mathbb{Z}_{>0}$.  One can define an action of a Kashiwara operator $\tf_{k,j}$ (for $k=0,\ldots,d-1$) as follows:
 $$\tf_{k,j}.\lambda=\mu \quad \text{ for all } \la\in\Pi$$
 where $\mu_{(t)}=\lambda_{(t)}$ when $t\neq d$ and $\mu_{(d)}=\tf_k \lambda_{(d)}$  (where $\tf_k$ 
 is denoting the usual Kashiwara operator acting on $d$-regular partitions). 
 
 Using the decomposition in Theorem \ref{3crystals}, we get an action of Kashiwara operators on the whole Fock space as follows. 
  Let $|\bla,\bs\rangle \in \Pi^\ell_s$ and write
  $$\beta (|\bla, \bs\rangle )=(|\bla_1,\br\rangle ,\sigma,|\bla_2,\dbr \rangle ).$$ 
  Then, for all $j\in \mathbb{Z}_{>0}$ :
  $$\tf_{k,j}. |\bla, \bs\rangle =\beta^{-1}(|\bla_1,\br\rangle,\tf_{k,j}.\sigma,|\bla_2,\dbr\rangle)$$
  For each $j\in \Z_{\geq0}$, we thus get an $\widehat{\mathfrak{sl}_d}$-crystal.
  
It is immediate to see that these actions  also commute with the 
$\sle$-crystal and the  $\sll$-crystal (it just follows from the existence of the bijection $\beta$). Finally, there is an obvious analogue of the Mullineux involution for this decomposition, which depends on $d$.   
Namely, for $|\bla,\bs\rangle \in \Pi^\ell_s$ and $\beta(|\bla,\bs\rangle)=(|\bla_1,\br\rangle ,\sigma,|\bla_2,\br \rangle )$, we define:
$$\Phi^{(d)} (|\bla,\bs\rangle  )= \beta^{-1} (m_{e,\br}(\bla_1) ,
m_e(\sigma_{(0)})+m_e(\sigma_{(1)})^d+\ldots +m_e(\sigma_{(n)})^{d^n}
,m_{\ell,\dbr}(\bla_2) ),$$
so that $\Phi^{(d)}$ generalizes simultaneously $\Phi$ and the
version of the Mullineux involution of \cite{DudasJacon2018} (which we recover by taking $\ell=1$).
As in \cite{DudasJacon2018}, 
we believe it would be interesting to look at the case $\ell=2$ and investigate the relationship between
$\Phi^{(d)}$ on the one hand, and the Alvis-Curtis duality for unipotent representations of
finite unitary groups in transverse characteristic $d$ on the other hand (or more generally of finite groups of Lie type $B$ and $C$).

\bibliographystyle{plain}

\end{document}